\documentclass{article}
\usepackage[english]{babel}
\usepackage[utf8]{inputenc}
\usepackage{johd}
\usepackage{mathptmx}
\usepackage{amsmath}
\usepackage{mathtools}
\usepackage{verbatim}
\usepackage{appendix}

\numberwithin{equation}{section}

\usepackage{amsmath,amsthm}       
\usepackage{mathrsfs}      
\usepackage{amssymb}       
\usepackage{amsfonts}      

\usepackage{cancel}
\usepackage{indentfirst}

\usepackage{graphicx}
\usepackage{subfigure}
\usepackage{wrapfig}  
\usepackage{multicol}
\usepackage{tikz}     

\usepackage[framemethod=TikZ]{mdframed}
\usepackage{color}
\usepackage{authblk}                          
\usepackage{textcomp}
\newtheorem{thm}{Theorem}[section]

\newtheorem{definition}[thm]{Definition}

\newtheorem{remark}[thm]{Remark}
\newtheorem{prop}[thm]{Proposition}

\author{Yiqing Lin}

\author{Kun Xu\footnote{Corresponding author. Email address: \url{1949101x_k@sjtu.edu.cn} (K. Xu).}}
\affil[1]{\small{School of Mathematical Sciences, Shanghai Jiao Tong University, 200240 Shanghai, China.}}

\begin{document}

\title{Particle systems for mean reflected BSDEs with jumps}

\date{April 2, 2024}
\maketitle
%
%
%
\begin{abstract}
In this paper, we study the mean reflected backward stochastic differential equations with jump (BSDEJs). We extend the work of Briand and Hibon on the propagation of chaos for mean reflected BSDEs \cite{briand2021particles} to the jump framework. Besides, we study the reflections for the particle system and obtain the rate of of convergence of the particle system towards the deterministic flat solution to the mean reflected BSDEJ.

\end{abstract}


\section{Introduction}
Pardoux and Peng \cite{pardoux1990adapted} first built the well-posedness of the nonlinear backward stochastic differential equations (BSDEs) driven by a Brownian motion with a Lipschitz continuous generator in 1990. After this seminal work, numbers of generalizations on the well-posedness of different types of BSDEs have been investigated. See \cite{kobylanski2000backward,briand2006bsde,briand2008quadratic,tevzadze2008solvability}. In order to solve the obstacle problem in partial differential equations, reflected backward stochastic differential equations (RBSDEs) was firstly introduced in El Karoui et al. \cite{Karoui1997ReflectedSO}. After that, 
Many studies related to RBSDEs have appeared, we refer the readers to Matoussi \cite{matoussi1997reflected}, Kobylanski et al. \cite{kobylanski2002reflected}, Lepeltier and Xu \cite{lepeltier2007reflected}, Bayrakstar and Yao \cite{bayraktar2012quadratic}, Essaky and Hassani \cite{essaky2011general}, Ren and Xia \cite{ren2006generalized}, Jia and Xu \cite{jia2008construction} and so on. 

In contrast to the pointwisely reflected BSDEs, mean reflected BSDEs were introduced in \cite{briand2018bsdes}. The formulation of the restricted BSDE reads,
\begin{equation}
\label{original}
\left\{\begin{array}{l}
 Y_t=\xi+\int_t^T f\left(s, Y_s, Z_s\right) d s-\int_t^T Z_s d W_s+K_T-K_t, \quad 0 \leq t \leq T; \\
\mathbb{E}\left[\ell\left(t, Y_t\right)\right] \geq 0, \quad 0 \leq t \leq T ,
\end{array}\right.
\end{equation}
where  $K$ is a deterministic process and $\ell$ is  a running loss function. The well-posedness of  such BSDEs with mean reflection was generalized in \cite{Hibon2017quadratic, hu2022general}, with quadratic generator and bounded or unbounded terminal condition. The following type of Skorokhod condition,
$$
\int_0^T \mathbb{E}\left[\ell\left(t, Y_t\right)\right] d K_t=0,
$$
ensures the existence and uniqueness of the so-called deterministic flat solution. This kind of BSDEs can help a lot when dealing with super-hedging problems under running risk management constraints. 

Apart from BSDEs in a Brownian framework, the generalizations of BSDEs to a setting with jumps enlarges the scope of applications of BSDEs, for instance in insurance modeling which is discussed in Liu and Ma \cite{Liu_2009}. Li and Tang \cite{Tang_1994} and Barles, Buckdahn and Pardoux \cite{barles1997backward} obtained a wellposedness result for Lipschitz BSDEs with jumps (BSDEJ), using a fixed point approach similar to that used in \cite{pardoux1990adapted}. Since then, studies on different kind of BSDEJs have appeared. BSDEJs driven by quadratic coefficients were studied by Becherer \cite{Becherer_2006} and Morlais \cite{morlais2010new} in an exponential utility maximization problem. Antonelli and Mancini \cite{antonelli2016solutions} gave the well-posedness of BSDEJs with local Lipschitz drivers. Besides, the teminal has a bounded condition in \cite{Becherer_2006,morlais2010new,antonelli2016solutions}. Similar results can be found in \cite{cohen2015stochastic,kazi2015quadratic,tevzadze2008solvability}. Barrieu and El Karoui \cite{Barrieu_2013} showed the existence of a solution with unbounded terminal under a quadratic structure condition in a continuous setup based on the stability of quadratic semimartingales. Moreover, studies on BSDEJs with a quadratic exponential structure can be found in \cite{ngoupeyou2010optimisation,jeanblanc2012robust,karoui2016quadratic,kaakai2022utility}. The well-posedness of BSDEs driven by general marked point processes were investigated in Confortola \& Fuhrman \cite{Confortola2013} for the weighted-$L^2$ solution,  Becherer \cite{Becherer_2006} and Confortola \& Fuhrman \cite{Confortola_2014} for the $L^2$ case,  Confortola,  Fuhrman \& Jacod \cite{Confortola2016} for the $L^1$ case and Confortola \cite{Confortola_2018} for the $L^p$ case. A more general BSDE with both Brownian motion diffusion term and a very general marked point process, which is non-explosive and has totally inaccessible jumps was studied in Foresta \cite{foresta2021optimal}. 

This paper generalizes the mean reflected BSDEs to a setting with jumps. Precisely, in this paper we study the following mean reflected BSDE with jump 
\begin{equation}
\label{eq 0}
\left\{\begin{array}{l}
Y_t =\xi+\int_t^T f\left(s, Y_s, U_s\right) ds -\int_t^T \int_E U_s(e) \Tilde{\mu}(d s, d e) + \left(K_T-K_t\right) , \quad \forall t \in[0, T] \text { a.s. };\\
\mathbb{E}\left[l\left(t, Y_t\right)\right] \geq 0, \quad \forall t \in[0, T].
\end{array}\right.
\end{equation}
We refer the reader to Section \ref{sec P} for more precise definitions and notations.

The theory of propagation of chaos can be traced back to the work by Kac \cite{kac1956foundations} whose initial
aim was to investigate the particle system approximation of some nonlocal partial differential equations (PDEs) arising in thermodynamics. Kac’s intuition was put into firm mathematical ground notably by Henry P McKean \cite{mckean1967propagation}, Alain-Sol Sznitman \cite{sznitman1991topics} and J{\"u}rgen G{\"a}rtner \cite{gartner1988mckean}. Further development and applications of propagation of chaos theory can be found in \cite{jabin2016mean,lacker2018strong,shkolnikov2012large}. Besides, Buckdahn et al. \cite{buckdahn2009mean}, Hu, Ren and Yang 
\cite{hu2023principal}, Laurière and Tangpi \cite{lauriere2022backward} and Briand et al. \cite{briand2020forward} studied the limit theorems for weakly interacting particles whose dynamics is given by a system of BSDEs in the case of non-reflected BSDE driven by Brownian motion. Li \cite{li2014reflected} extends the results of \cite{buckdahn2009mean} to reflected BSDEs where the weak interaction enters only the driver, while the work by Briand and Hibon \cite{briand2021particles} considers a particular class of mean reflected BSDEs. We extend in this paper the work of Briand and Hibon on the propagation of chaos for mean reflected BSDEs \cite{briand2021particles} to the jump framework.

The paper is organized as follows. In Section \ref{sec P}, we introduce some basic notations. Section 3 contains the well-posedness of (\ref{eq 0}) and some regularity result on $K$ and $Y$. Finally, in the last section, we prove that BSDEs coming from these particles systems have a unique solution and we prove that the solution of this system converges to the solution of the mean reflected BSDEJ. We give also the rate of convergence of the propagation of chaos.

\section{Preliminaries}
\label{sec P}
\subsection{General setting}
Let $\left(\Omega, \mathcal{F},\left\{\mathcal{F}_t\right\}_{0 \leq t \leq T}, \mathbb{P}\right)$ be a filtered probability space, whose filtration satisfies the usual hypotheses of completeness and right-continuity. We suppose that this filtration is generated by by the following two mutually independent processes:
\begin{itemize}
    \item a $d$-dimensional standard Brownian motion $\left\{W_t\right\}_{t \geq 0}$, and
    \item a Poisson random measure $\mu$ on $\mathbb{R}_{+} \times E$, where $E \triangleq \mathbb{R}^{\ell} \backslash\{0\}$ is equipped with its Borel field $\mathcal{E}$, with compensator $\lambda(\omega, d t, d e)$. We assume in all the paper that $\lambda$ is absolutely continuous with respect to the Lebesgue measure $d t$, i.e. $\lambda(\omega, d t, d e)=\nu_t(\omega, d e) d t$. Finally, we denote $\tilde{\mu}$ the compensated jump measure
    $$
    \widetilde{\mu}(\omega, d e, d t)=\mu(\omega, d e, d t)-\nu_t(\omega, d e) d t .
    $$
\end{itemize}

Denote by $\mathbb{F}:=\left\{\mathcal{F}_t\right\}_{t \in[0, T]}$ the completion of the filtration generated by $\tilde{\mu}$. 
Following Li and Tang \cite{Tang_1994} and Barles, Buckdahn and Pardoux \cite{barles1997backward}, the definition of a BSDE with jumps is then
\begin{definition}
Let $\xi$ be a $\mathcal{F}_T$-measurable random variable. A solution to the BSDEJ with terminal condition $\xi$ and generator $f$ is a triple $(Y, Z, U)$ of progressively measurable processes such that
\begin{equation}
\label{eq defn}
Y_t=\xi+\int_t^T f_s\left(Y_s, Z_s, U_s\right) d s-\int_t^T Z_s d B_s-\int_t^T \int_E U_s(x) \widetilde{\mu}(d s, d e), t \in[0, T], \mathbb{P}-a . s .,
\end{equation}
where $f: \Omega \times[0, T] \times \mathbb{R} \times \mathbb{R}^d \times \mathcal{A}(E) \rightarrow \mathbb{R}$ is a given application and
$$
\mathcal{A}(E):=\{u: E \rightarrow \mathbb{R}, \mathcal{B}(E)-\text { measurable }\} .
$$

Then, the processes $Z$ and $U$ are supposed to satisfy the minimal assumptions so that the quantities in (\ref{eq defn}) are well defined:
$$
\int_0^T\left|Z_t\right|^2 d t<+\infty,\left(\operatorname{resp} . \int_0^T \int_E\left|U_t(x)\right|^2 \nu_t(d x) d t<+\infty\right), \mathbb{P}-\text { a.s. }
$$
\end{definition}

As the Brownian motion $B$ and the integer valued random measure $\mu$ is independent, proof related to the two does not interfere with each other. In order to highlight the key points of this article, we only consider equation without the Brownian motion term like (\ref{eq 0}).

\subsection{Notation}

We denote a generic constant by $C$, which may change line by line, is sometimes associated with several subscripts (such as $C_{K,T}$ ) showing its dependence when necessary. Let us introduce the following spaces for stochastic processes:
\begin{itemize}
    \item For any real $p \geq 1, \mathcal{S}^p$ denotes the set of real-valued, adapted and c\`adl\`ag processes $\left\{Y_t\right\}_{t \in[0, T]}$ such that
    $$
    \|Y\|_{\mathcal{S}^p}:=\mathbb{E}\left[\sup _{0 \leq t \leq T}\left|Y_t\right|^p\right]^{1 / p}<+\infty.
    $$
    Then $\left(\mathcal{S}^p,\|\cdot\|_{ \mathcal{S}^p}\right)$ is a Banach space.
    \item $\mathscr{M}^{2,p}$ the set of predictable processes $U$ such that $$\|U\|_{\mathscr{M}^{2,p}}:=\left(\mathbb{E}\left[\int_{[0, T]} \int_E\left|U_s(e)\right|^2 \nu_s(de) d s\right ]^{\frac{p}{2}}\right)^{\frac{1}{p}}<\infty.$$
    \item $\mathcal{S}^\infty$ is the space of $\mathbb{R}$-valued c\`adl\`ag and $\mathbb F$-progressively measurable processes $Y$ such that
    $$
    \|Y\|_{\mathcal{S}^\infty}:=\left\|\sup _{0 \leq t \leq T}|Y_t|\right\|_{\infty}<+\infty.
    $$
    \item $\mathcal{J}^{\infty}$ is the space of functions which are $d \mathbb{P} \otimes v(d z)$ essentially bounded i.e.,
    $$
    \|\psi\|_{\mathcal{J}^{\infty}}:=\left\|\sup _{t \in[0, T]}\| \psi_t\|_{\mathcal{L}^{\infty}(v)}\right\|_{\infty}<\infty,
    $$
    where $\mathcal{L}^{\infty}(v)$ is the space of $\mathbb{R}^k$-valued measurable functions $v(d z)$-a.e. bounded endowed with the usual essential sup-norm.
    \item $\mathcal{P}_p(\mathbb{R})$ is the collection of all probability measures over $(\mathbb{R}, \mathcal{B}(\mathbb{R}))$ with finite $p^{\text {th }}$ moment, endowed with the $p$-Wasserstein distance $W_p$;
    \item $\mathcal{A}^2$ the closed subset of $\mathcal{S}^2$ consisting of non-decreasing processes starting from 0.
    \item $\mathcal{A}^D$ is the closed subset of $\mathcal{S}^{\infty}$ consisting of non-decreasing processes $K=\left(K_t\right)_{0 \leq t \leq T}$ starting from the origin, i.e. $K_0 = 0$;
    \item $\mathcal{A}_D$ is the space of all càdlàg non-decreasing deterministic processes $K=\left(K_t\right)_{0 \leq t \leq T}$ starting from the origin;
    \item $\mathcal{A}_D^\infty$ is the closed subset of $\mathcal{S}^{\infty}$ consisting of deterministic non-decreasing processes $K=\left(K_t\right)_{0 \leq t \leq T}$ starting from the origin.
\end{itemize}

For $\beta>0$, we introduce the following spaces.
\begin{itemize}
    \item $L^{r, \beta}(s)$ is the space of all $\mathbb{F}$-progressive processes $X$ such that
    $$
    \|X\|_{L^{r, \beta}(s)}^r=\mathbb{E}\left[\int_0^T e^{\beta s}\left|X_s\right|^r d s\right]<\infty .
    $$
    \item $L^{r, \beta}(p)$ is the space of all $\mathbb{F}$-predictable processes $U$ such that
    $$
    \|U\|_{L^{r, \beta}(p)}^r=\mathbb{E}\left[\int_0^T \int_E e^{\beta s}\left|U_s(e)\right|^r \phi_s(d e) d s\right]<\infty.
    $$
\end{itemize}

\section{Well-posedness and some regularity results}
\label{sec 3}

In this section, we consider the following mean reflected BSDEJ:
\begin{equation}
\label{eq00}
\left\{\begin{array}{l}
Y_t =\xi+\int_t^T f\left(s, Y_s, U_s\right) ds -\int_t^T \int_E U_s(e) \Tilde{\mu}(d s, d e) + \left(K_T-K_t\right) , \quad \forall t \in[0, T], \quad \mathbb{P} \text {-a.s. };\\
\mathbb{E}\left[l\left(t, Y_t\right)\right] \geq 0, \quad \forall t \in[0, T] .
\end{array}\right.
\end{equation}

These parameters are supposed to satisfy the following standard running assumptions:

$\left(H_{\xi}\right)$ The terminal condition $\xi$ is an $\mathcal{F}_T$-measurable random variable bounded by some constant $M>0$ such that
$$
\mathbb{E}[l(T, \xi)] \geq 0 .
$$

$\left(H_f\right)$ The map $(\omega, t) \mapsto f(\omega, t, \cdot)$ is $\mathbb{F}$-progressively measurable.

(1) (Bounded condition) For each $t \in[0, T]$, $f(t, 0,0)$ is bounded by some constant $L$, $\mathbb{P}$-a.s.

(2) (Lipschitz condition) There exists $\lambda\geq 0$, such that for every $\omega \in \Omega,\ t \in[0, T],\ y_1, y_2 \in \mathbb{R}$,\ $u_1,u_2\in L^2(\mathcal{B}(E),\nu_t, \mathbb{R})$, we have
$$
\begin{aligned}
& \left|f(\omega, t, y_1, u_1)-f\left(\omega, t, y_2, u_2\right)\right| \leq \lambda\left(\left|y_1-y_2\right|+|u_1-u_2|_\nu\right).
\end{aligned}
$$

(3) (Growth condition) For all $t \in[0, T], \ (y,u) \in \mathbb{R} \times  L^2(\mathcal{B}(E),\nu_t; \mathbb R):\ \mathbb{P}$-a.s, there exists $\lambda>0$ such that,
$$
-\frac{1}{\lambda} j_{\lambda}(t,- u)-\alpha_t-\beta|y| \leq f(t, y, u) \leq \frac{1}{\lambda} j_{\lambda}(t, u)+\alpha_t+\beta|y|,
$$
where $\{\alpha_t\}_{0 \leq t \leq T}$ is  a progressively measurable nonnegative stochastic process with $\|\alpha\|_{\mathcal{S}^{\infty}}<\infty$. 

(4) ($A_{\gamma}$-condition) For all $t \in[0, T], L>0,\  y \in \mathbb{R},\   u_1, u_2 \in$ $\mathbb{L}^2\left(\mathcal{B}(E),\nu_t; \mathbb{R}\right),$ with $|y|,\ \|u_1\|_{\mathcal{J}^\infty},\ \left\|u_2\right\|_{\mathcal{J}^\infty} \leq L$, there exists a $\mathcal{P} \otimes \mathcal{E}$-measurable process $\gamma^{y, u_1, u_2}$ satisfying $d t \otimes d \mathbb{P}$-a.e.
$$
f(t, y, u_1,\mu)-f\left(t, y, u_2,\mu\right) \leq \int_E \gamma_t^{y, u_1,u_2}(x)\left[u_1(x)-u_2(x)\right] \nu(d x)
$$
and $C_L^1(1 \wedge|x|) \leq \gamma_t^{y,  u_1, u_2}(x) \leq C_L^2(1 \wedge|x|)$ with two constants $C_L^1, C_L^2$. Here, $C_L^1>-1$ and $C_L^2>0$ depend on $L$. (Hereafter, we frequently omit the superscripts to lighten the notation.) ($\mathcal{P}$ denotes the $\sigma$ algebra of progressively measurable sets of $[0, T] \times$ $\Omega$, respectively)
\hspace*{\fill}\\

$\left(H_{l}\right)$ The running loss function $l: \Omega \times[0, T] \times \mathbb{R} \rightarrow \mathbb{R}$ is an $\mathscr{F}_T \times \mathcal{B}(\mathbb{R}) \times \mathcal{B}(\mathbb{R})$ measurable map and there exists some constant $C>0$ such that, $\mathbb{P}$-a.s.,

1. $(t, y) \rightarrow l(t, y)$ is continuous,

2. $\forall t \in[0, T], y \rightarrow l(t, y)$ is strictly increasing,

3. $\forall t \in[0, T], \mathbb{E}[l(t, \infty)]>0$,

4. $\forall t \in[0, T], \forall y \in \mathbb{R},|l(t, y)| \leq C(1+|y|)$.
\hspace*{\fill}\\

In order to introduce another assumption for the main result of this paper, we define the operator $L_t: L^{2, \beta}(s) \rightarrow[0, \infty)$, $t \in[0, T]$ by
$$
L_t: X \rightarrow \inf \{x \geq 0: \mathbb{E}[l(t, x+X)] \geq 0\}
$$
which is well-defined due to Assumption $\left(H_{l}\right)$, see  \cite{briand2018bsdes}. Since $L_t(0)$ is continuous in $t$, without loss of generality we assume that $\left|L_t(0)\right| \leq L$ for each $t \in[0, T]$, see also \cite{briand2018bsdes}.
\hspace*{\fill}\\

$\left(H^{'}_{l}\right)$ There exist two constants $\overline\kappa>\underline\kappa>0$ such that for $\forall y_1, y_2 \in \mathbb{R}$,
$$
\underline{\kappa}\left|y_1-y_2\right| \leq\left|l\left(t, y_1\right)-l\left(t, y_2\right)\right| \leq \overline\kappa\left|y_1-y_2\right| .
$$
\hspace*{\fill}\\

$\left(H^{''}_{l}\right)$ There exist a constant $r>0$ such that for $\forall s, t \in [0, T]$,
$$
\left|l\left(t, y\right)-l\left(s, y\right)\right| \leq r\left|t-s\right|.
$$

\begin{remark}
\label{rmk L}
If assumption $\left(H^{'}_{l}\right)$ is also fulfilled, then for each $t \in[0, T]$, $\kappa:=\overline\kappa/\underline\kappa>1$,
\begin{equation}
\label{eqL}
\left|L_t\left(\eta^1\right)-L_t\left(\eta^2\right)\right| \leq \kappa \mathbb{E}\left[\left|\eta^1-\eta^2\right|\right], \quad \forall \eta^1, \eta^2 \in L^{2, \beta}(s).
\end{equation}
Besides, if Assumption $\left(H^{''}_{l}\right)$ also holds, we can deduce that
\begin{equation}
\label{eqLL}
\left|L_t\left(\eta^1\right)-L_s\left(\eta^2\right)\right| \leq \kappa \mathbb{E}\left[\left|\eta^1-\eta^2\right|\right] + \bar{r} |t-s|, \quad \forall t, s \in [0,T], \quad \forall \eta^1, \eta^2 \in L^{2, \beta}(s),
\end{equation}
where $\bar{r} = r/\underline\kappa$.

Indeed, 
$$
\begin{aligned}
\left|L_t\left(\eta^1\right)-L_s\left(\eta^2\right)\right| &\leq \left| L_t(\eta^1)-L_t(\eta^2)\right| + \left| L_t(\eta^2)-L_s(\eta^2)\right| \\
&\leq \kappa \mathbb{E}\left[\left|\eta^1-\eta^2\right|\right] + \left| L_t(\eta^2)-L_s(\eta^2)\right|.
\end{aligned}
$$
By the definition of function $L$, we know that $\left| L_t(\eta^2)-L_s(\eta^2)\right|$ means: the random variable $\eta^2$ in $L$ is the same, as the time changes, $\mathbb{E}\left[l(t, \eta^2) \right]$ changes as well and the change is $\left|\mathbb{E}\left[l\left(t, \eta^2\right)-l\left(s, \eta^2\right) \right]\right|$. If we keep the time consistent, then how many transformations do we need on the random variable to make up for this change. 
From Assumption $\left(H^{''}_{l}\right)$, we know that
$$
\left|\mathbb{E}\left[l\left(t, \eta^2\right)-l\left(s, \eta^2\right) \right]\right| \leq r\left|t-s\right|.
$$
Thus, assume $L_t(\eta^2)=y_1, L_s(\eta^2)=y_2$, we have
$$
\underline\kappa \left|y_1-y_2\right| \leq \left|\mathbb{E}\left[l\left(t, y_1+\eta^2\right)-l\left(t, y_2+\eta^2\right) \right]\right| \leq r |t-s|.
$$
Hence, we can deduce that $\left|y_1-y_2\right| \leq r/\underline\kappa |t-s|$, then we have
$$
\left|L_t\left(\eta^1\right)-L_s\left(\eta^2\right)\right| \leq \kappa \mathbb{E}\left[\left|\eta^1-\eta^2\right|\right] + \bar{r} |t-s|.
$$
\end{remark}

\subsection{Well-posedness}
We first recall the previous conclusions about the well-posedness of mean reflected BSDEJ (\ref{eq00}). \cite{gu2023mean} gives the existence and uniqueness of the solution of this equation in the space $L^{2, \beta}(s) \times L^{2, \beta}(p) \times \mathcal{A}_D$ under certain conditions. The purpose of this subsection is to prove that (\ref{eq00}) has a unique solution $(Y,U,K)$ in the space $\mathcal{S}^\infty \times \mathcal{J}^\infty \times \mathcal{A}_D^\infty$. This result is not obvious. Just like the counterexample given in Remark 2.7 of \cite{barles1997backward}, under the assumptions given in this paper, the comparison theorem for BSDEJ and Girsanov transformation in this setting are not valid, so it is difficult for us to directly deduce the bounded space of the solution according to the bounded condition of the terminal and generator.

\begin{prop}
\label{prop standard}
Assume that:

(i)The map $(\omega, t) \mapsto f(\omega, t, \cdot)$ is $\mathbb{F}$-progressively measurable. There exist a positive $\mathbb{F}$-progressively measurable process $\left(\alpha_t, t \in[0, T]\right)$ such that $-\frac{1}{\lambda} j_{\lambda}(t,- u)-\alpha_t\leq f(t,u) \leq \alpha_t+\frac{1}{\lambda} j_{\lambda}(t,u)$ $d t \otimes d \mathbb{P}$-a.e. $(\omega, t) \in \Omega \times[0, T]$.

(ii) $|\xi|,\left(\alpha_t, t \in[0, T]\right)$ are essentially bounded, i.e., $\|\xi\|_{\infty},\|\alpha\|_{\mathcal{S}^{\infty}}<\infty$.

(iii) There exists $\lambda\geq 0$, such that for every $\omega \in \Omega,\ t \in[0, T]$,\ $u_1,u_2\in L^2(\mathcal{B}(E),\nu_t, \mathbb{R})$, we have
$$
\begin{aligned}
& \left|f(\omega, t, u_1)-f\left(\omega, t, u_2\right)\right| \leq \lambda|u_1-u_2|_\nu.
\end{aligned}
$$

(iv) For all $t \in[0, T], M>0,\   u_1, u_2 \in$ $\mathbb{L}^2\left(\mathcal{B}(E),\nu_t; \mathbb{R}\right),$ with $\|u_1\|_{\mathcal{J}^\infty},\ \left\|u_2\right\|_{\mathcal{J}^\infty} \leq M$, there exists a $\mathcal{P} \otimes \mathcal{E}$-measurable process $\gamma^{u_1, u_2}$ satisfying $d t \otimes d \mathbb{P}$-a.e.
$$
f(t, u_1,\mu)-f\left(t, u_2,\mu\right) \leq \int_E \gamma_t^{u_1,u_2}(x)\left[u_1(x)-u_2(x)\right] \nu(d x)
$$
and $C_M^1(1 \wedge|x|) \leq \gamma_t^{u_1, u_2}(x) \leq C_M^2(1 \wedge|x|)$ with two constants $C_M^1, C_M^2$. Here, $C_M^1>-1$ and $C_M^2>0$ depend on $M$.

Then BSDE
\begin{equation}
\label{eq standard}
\bar{Y}_t= \xi + \int_t^T f(s,\bar{U}_s) ds -\int_t^T \int_E \bar{U}_s(e) \tilde{\mu}(ds,de)
\end{equation}
has a unique solution $(\bar{Y},\bar{U})$ in the space $\mathcal{S}^{\infty} \times \mathcal{J}^{\infty}$.
\end{prop}

\begin{proof}
The well-posedness can be deduced from \cite{fujii2018quadratic}.
\end{proof}

We first consider a simple case where the generator has the same structure as (\ref{eq standard}):
\begin{equation}
\label{eq fixed}
\left\{\begin{array}{l}
\tilde{Y}_t=\xi+\int_t^T f(s, \tilde{U}_s) d s-\int_t^T\int_E \tilde{U}_s(e) \tilde{\mu}(ds,de)+\tilde{K}_T-\tilde{K}_t , \quad \forall t \in[0, T] , \quad \mathbb{P} \text {-a.s. };\\
\mathbb{E}\left[l\left(t,\tilde{Y}_t\right)\right] \geq 0, \quad \forall t \in [0,T].
\end{array}\right. 
\end{equation}

\begin{thm}
\label{thm fixed mean reflected}
Assume that the same conditions as Proposition \ref{prop standard} hold. Then the BSDE (\ref{eq fixed}) with mean reflection has a unique deterministic flat solution $(\tilde{Y}, \tilde{U}, \tilde{K})$ in the space $ \mathcal{S}^{\infty} \times \mathcal{J}^\infty \times \mathcal{A}_D^\infty$.  
\end{thm}

\begin{proof}
Obviously, the generator $f(s, \tilde{U}_s)$, the terminal $\xi$ and the loss function $l$ satisfies the conditions of Theorem 5.3 in \cite{gu2023mean}, thus BSDE (\ref{eq fixed}) admits a unique deterministic flat solution $(\tilde{Y}, \tilde{U}, \tilde{K}) \in L^{2, \beta}(s) \times L^{2, \beta}(p) \times \mathcal{A}_D$.

Moreover, for each $t \in[0, T]$ we have
$$
\tilde{K}_t=\sup _{0 \leq s \leq T} L_s\left(\bar{Y}_s\right)-\sup _{t \leq s \leq T} L_s\left(\bar{Y}_s\right) \text { and } \bar{Y}_t=\mathbb{E}_t\left[\xi+\int_t^T f(s,\bar{U}_s) d s\right] .
$$

Consequently, $\left(\tilde{Y}_t-\left(\tilde{K}_T-\tilde{K}_t\right), \tilde{U}_t\right)_{0 \leq t \leq T}$ and $\left(\bar{Y}_t, \bar{U}_t\right)_{0 \leq t \leq T}$ are both solutions to the following standard BSDE on the time interval $[0, T]$,
$$
\hat{Y}_t=\xi+\int_t^T f(s, \hat{U}_s) d s-\int_t^T \int_E \hat{U}_s(e) \tilde{\mu}(ds,de) .
$$

By the uniqueness of solutions to BSDE, we deduce that
$$
\left(\bar{Y}_t, \bar{U}_t\right)=\left(\tilde{Y}_t-\left(\tilde{K}_T-\tilde{K}_t\right), \tilde{U}_t\right), \forall t \in[0, T] .
$$

Since $\tilde{K}$ is a deterministic continuous process, we have $\tilde{Y} \in \mathcal{S}^{\infty}$. Thus $(\tilde{Y}, \tilde{U}, \tilde{K}) \in \mathcal{S}^{\infty} \times \mathcal{J}^\infty \times \mathcal{A}_D^\infty$ is the unique deterministic flat solution to the BSDE (\ref{eq fixed}) with mean reflection.
\end{proof}

\begin{remark}
\label{rmk rep}
In view of the proof of Theorem \ref{thm fixed mean reflected}, we can introduce a representation result which plays a key role in establishing the existence and uniqueness result. Indeed, we can prove this representation result in a way similar to that of Lemma 5.4 in \cite{gu2023mean}.  

Assume assumptions $\left(H_{\xi}\right)$, $\left(H_f\right)$, and $\left(H_l\right)$ hold. Suppose $(Y, U, K) \in \mathcal{S}^{\infty} \times \mathcal{J}^\infty \times \mathcal{A}_D^\infty$ is a deterministic flat solution to the BSDE with mean reflection (\ref{eq00}). Then, for each $t \in[0, T]$
$$
\left(Y_t, U_t, K_t \right)=\left(y_t+\sup _{t \leq s \leq T} L_s\left(y_s\right), u_t, \sup _{0 \leq s \leq T} L_s\left(y_s\right)-\sup _{t \leq s \leq T} L_s\left(y_s\right)\right),
$$
where $(y, u) \in \mathcal{S}^{\infty} \times J^\infty $ is the solution to the following BSDE with the driver $f\left(s, Y_s,u_s\right)$ on the time horizon $[0, T]$ , and $Y \in \mathcal{S}^{\infty}$ is fixed by the solution of (\ref{eq00}):
\begin{equation}
{\label{eq samll}}
y_t=\xi+\int_t^T f\left(s, Y_s,u_s\right) d s-\int_t^T \int_E u_s(e) \tilde{\mu}(d s, d e).
\end{equation}

\end{remark}

Next, we study the BSDE (\ref{eq00}) with mean reflection. We first prove the existence and uniqueness of the solution on a small time interval $[T-h, T]$, which is called local solution. Then we stitch local solutions to build the global solution.

\begin{thm}
\label{thm well-posedness}
Assume assumptions $\left(H_{\xi}\right)$, $\left(H_f\right)$,  $\left(H_l\right)$ and $\left(H^{'}_l\right)$ hold. Then the BSDE (\ref{eq00}) with mean reflection has a unique deterministic flat solution $(Y, U, K)$ in the space $\mathcal{S}^{\infty} \times \mathcal{J}^\infty \times \mathcal{A}_D^\infty$. 
\end{thm}

\begin{proof}
\textbf{Step 1.} Define the solution map. Suppose $\left(H_{\xi}\right)$, $\left(H_f\right)$ and $\left(H_h\right)$ hold, we can deduced that for each $P \in \mathcal{S}_{[T-h, T]}^{\infty}$, it follows from Theorem \ref{thm fixed mean reflected} that the following BSDE with mean reflection
\begin{equation}
\label{eq fixed with mean}
\left\{\begin{array}{l}
Y_t^P=\xi+\int_t^T f\left(s, P_s,  U_s^P\right) d s-\int_t^T\int_E U_s^P(e) \tilde{\mu}(ds,de)+K_T^P-K_t^P , \quad \forall t \in[T-h, T], \quad \mathbb{P} \text {-a.s. };\\
\mathbb{E}\left[l\left(t, Y_t^P\right)\right] \geq 0, \quad \forall t \in [T-h, T].
\end{array}\right.
\end{equation}
has a unique deterministic flat solution $\left(Y^P, U^P, K^P\right) \in \mathcal{S}_{[T-h, T]}^{\infty} \times \mathcal{J}^{\infty}_{[T-h, T]} \times$ $\mathcal{A}_{[T-h, T]}^D$. Then we define the solution map $\Gamma: P \rightarrow \Gamma(P)$ by
$$
\Gamma(P):=Y^P, \quad \forall P \in \mathcal{S}_{[T-h, T]}^{\infty} .
$$

In order to show that $\Gamma$ is contractive, for each real number $A \geq A_0$ we consider the following set:
$$
\mathscr{B}_A:=\left\{P \in \mathcal{S}_{[T-h, T]}^{\infty}:\|P\|_{\mathcal{S}_{[T-h, T]}^{\infty}} \leq A\right\},
$$
where
$$
A_0=(3+2\kappa+4\lambda \kappa)L.
$$

\textbf{Step 2.} Prove that if $A \geq A_0$, then there is a constant $\delta^A>0$ depending only on $L, \lambda, \kappa$ and $A$ such that for any $h \in\left(0, \delta^A\right],\Gamma\left(\mathscr{B}_A\right) \subset \mathscr{B}_A$.

In view of Remark \ref{rmk rep}, we conclude that for each $t \in[T-h, T]$,
$$
\left(Y_t^P, U_t^P\right)=\left(y_t^P+\left(K_T^P-K_t^P\right), u_t^P\right),
$$
where $\left(y^P, u^P\right) \in \mathcal{S}_{[T-h, T]}^{\infty} \times \mathcal{J}^{\infty}_{[T-h, T]}$ is the solution to the following standard BSDE on the time interval $[T-h, T]$,
\begin{equation}
\label{eq fixed no mean}
y_t^P=\xi+\int_t^T f\left(s, P_s,u_s^P\right) d s-\int_t^T \int_E u_s^P(e) \tilde{\mu}(ds,de),
\end{equation}
and for each $t \in[T-h, T]$,
$$
K_T^P-K_t^P=\sup _{t \leq s \leq T} L_s\left(y_s^P\right),
$$
where
\begin{equation}
\label{bar Y}
\begin{aligned}
y_t^P&=\mathbb{E}_t\left[\xi+\int_t^T f\left(s, P_s,u_s^P\right) d s\right] \\
&= \mathbb{E}_t\left[\xi+\int_t^T f\left(s, P_s,u_s^P\right)-f\left(s, P_s,0\right)+f\left(s, P_s,0\right) d s\right]\\
&\leq \mathbb{E}_t^{\tilde{\mu}^\gamma}\left[\xi+\int_t^T |f\left(s, P_s,0\right)| d s\right] \leq L+(L+\lambda A)h,    
\end{aligned}
\end{equation}
where $\tilde{\mu}^\gamma := \tilde{\mu} -<\tilde{\mu}, \tilde{\gamma}\cdot \tilde{\mu}> $, $\tilde{\gamma} = \gamma^{u_s^P}$, and the first inequality is deduced from the $A_\gamma$ condition, the second inequality is deduced from the Lipschitz condition.

Consequently, we obtain that
\begin{equation}
\label{Y}
\left\|Y^P\right\|_{\mathcal{S}_{[T-h, T]}^{\infty}} \leq\left\|y^P\right\|_{\mathcal{S}_{[T-h, T]}^{\infty}}+\sup _{T-h \leq s \leq T} L_s\left(y_s^P\right) .
\end{equation}

Thanks to Remark \ref{rmk L}, for each $s \in[T-h, T]$ we have
$$
\left|L_s\left(y_s^P\right)-L_s(0)\right| \leq \kappa \mathbb{E}\left[\left|y_s^P\right|\right].
$$

Therefore from Assumption $\left(H_f\right)$ we deduce that
\begin{equation}
\label{Y1}
\begin{aligned}
\sup _{T-h \leq s \leq T} L_s\left(y_s^P\right) & \leq L+\kappa \sup _{T-h \leq s \leq T} \mathbb{E}\left[|\xi|+\int_s^T\left(\left|f\left(r, 0,0\right)\right|+\lambda A + \lambda \left(\int_E |u_r^P(e)|^2 \nu(de)\right)^{\frac{1}{2}}\right) d r\right] \\
& \leq(\kappa+1) L+\kappa h (L+\lambda A)+ 2\kappa \lambda \|y^P\|_{\mathcal{S}^\infty}\\
&\leq (\kappa+1) L+\kappa h (L+\lambda A)+ 2\kappa \lambda (L+(L+\lambda A)h),
\end{aligned}
\end{equation}
in view of Corollary 1 from \cite{morlais2009utility}, we can have the second inequality.

Then we define
$$
\delta^A:=\min \left(\frac{L}{L+\lambda A}, T\right).
$$

Recalling equations (\ref{bar Y}), (\ref{Y}) and (\ref{Y1}), we derive that for each $h \in\left(0, \delta^A\right]$,
$$
\left\|Y^P\right\|_{\mathcal{S}_{[T-h, T]}^{\infty}} \leq (3+2\kappa+4\lambda \kappa)L=A_0 \leq A,
$$
which is the desired result.

\textbf{Step 3.} Now we are going to show the contractive property of the solution map $\Gamma$.

For $\forall P^1, P^2 \in \mathcal{B}_A$, set
$$
Y^{P^i}=\Gamma\left(P^i\right), \quad i=1,2.
$$
where $\left(Y^{P^i}, U^{P^i}, K^{P^i}\right)$ is the solution to the BSDE (\ref{eq fixed with mean}) with mean reflection associated with the data $P^i$. Applying Theorem \ref{thm fixed mean reflected} again, we conclude that for each $t \in[T-h, T]$,
$$
\left(Y_t^{P^i}, U_t^{P^i}\right)=\left(y_t^{P^i}+\left(K_T^{P^i}-K_t^{P^i}\right), u_t^{P^i}\right),
$$
where $\left(y^{P^i}, u^{P^i}\right) \in \mathcal{S}_{[T-h, T]}^{\infty} \times \mathcal{J}^\infty_{[T-h, T]}$ is the solution to the BSDE (\ref{eq fixed no mean}) associated with the data $P^i$. Since for each $t \in[T-h, T]$,
$$
K_T^{P^i}-K_t^{P^i}=\sup _{t \leq s \leq T} L_s\left(y_s^{P^i}\right) \text { and } y_t^{P^i}=\mathbb{E}_t\left[\xi+\int_t^T f\left(s, P_s^i, u_s^{P^i}\right) d s\right],
$$
we have
$$
\begin{aligned}
\sup _{T-h \leq t \leq T}\left|\left(K_T^{P^1}-K_t^{P^1}\right)-\left(K_T^{P^2}-K_t^{P^2}\right)\right| &\leq  \sup _{T-h \leq s \leq T}\left|L_s\left(y_s^{P^1}\right)-L_s\left(y_s^{P^2}\right)\right| \\
&\leq \kappa \sup _{T-h \leq s \leq T} \mathbb{E}\left[ \left|y^{P^1}_s- y^{P^2}_s\right|\right]\\
&\leq \kappa \sup _{T-h \leq s \leq T} \mathbb{E}\left[ \int_s^T \left| f(r,P_r^1,u_r^{P^1})-f(r,P_r^1,u_r^{P^2})+f(r,P_r^1,u_r^{P^2})-f(r,P_r^2,u_r^{P^2})\right|dr\right]\\
&\leq \kappa \sup _{T-h \leq s \leq T} \mathbb{E}^{\tilde{\mu}^\gamma}\left[ \int_s^T \left|f(r,P_r^1,u_r^{P^2})-f(r,P_r^2,u_r^{P^2})\right|dr\right]\\
&\leq \kappa \lambda h \left\|P^1-P^2 \right\|_{\mathcal{S}^\infty_{[T-h,T]}},
\end{aligned}
$$
where $\tilde{\mu}^\gamma := \tilde{\mu} -<\tilde{\mu}, \tilde{\gamma}\cdot \tilde{\mu}> $, $\tilde{\gamma} = \gamma^{u_r^{P^1},u_r^{P^2}}$. 

Similarly,
$$
\begin{aligned}
\left|y^{P^1}_s- y^{P^2}_s\right| &\leq \mathbb{E}_t\left[ \int_s^T \left| f(r,P_r^1,u_r^{P^1})-f(r,P_r^1,u_r^{P^2})+f(r,P_r^1,u_r^{P^2})-f(r,P_r^2,u_r^{P^2})\right|dr\right]\\
&\leq \mathbb{E}_t^{\tilde{\mu}^\gamma}\left[ \int_s^T \left|f(r,P_r^1,u_r^{P^2})-f(r,P_r^2,u_r^{P^2})\right|dr\right].
\end{aligned}
$$
Hence we have
$$
\left\|y^{P^1}-y^{P^2}\right\|_{\mathcal{S}_{[T-h, T]}^{\infty}} \leq \kappa \lambda h \left\|P^1-P^2 \right\|_{\mathcal{S}^\infty_{[T-h,T]}}.
$$

Thus we can conclude that
$$
\begin{aligned}
\left\|Y^{P^1}-Y^{P^2}\right\|_{\mathcal{S}_{[T-h, T]}^{\infty}}
&\leq  \left\|y^{P^1}-y^{P^2}\right\|_{\mathcal{S}_{[T-h, T]}^{\infty}}+\sup _{T-h \leq t \leq T}\left|\left(K_T^{P^1}-K_t^{P^1}\right)-\left(K_T^{P^2}-K_t^{P^2}\right)\right| \\
&\leq \left\|y^{P^1}-y^{P^2}\right\|_{\mathcal{S}_{[T-h, T]}^{\infty}}+\kappa \lambda h \left\|P^1-P^2 \right\|_{\mathcal{S}^\infty_{[T-h,T]}}\\
&\leq 2\kappa\lambda h \left\|P^1-P^2 \right\|_{\mathcal{S}^\infty_{[T-h,T]}}.
\end{aligned}
$$

Now we define
$$
\hat{\delta}^A:=\min \left(\frac{1}{4 \lambda \kappa}, \delta^A\right),
$$
and it is straightforward to check that for any $h \in\left(0, \hat{\delta}^A\right]$,
$$
\left\|Y^{P^1}-Y^{P^2}\right\|_{\mathcal{S}_{[T-h, T]}^{\infty}} \leq \frac{1}{2}\left\|P^1-P^2\right\|_{\mathcal{S}_{[T-h, T]}^{\infty}}.
$$
Therefore, $\Gamma$ defines a strict contraction map on the time interval $[T-h,T]$, which implies the existence and uniqueness of local solution on $[T-h,T]$ due to the completeness of the space.

\textbf{Step 4.} Next we stitch the local solutions to get the global solution on $[0,T]$. More precisely, choose  $h \in\left(0, \hat{\delta}^A\right]$ and $T=n h$ for some integer $n$. Then  the BSDE (\ref{eq00}) with mean reflection admits a unique deterministic flat solution $\left(Y^n, U^n, K^n\right) \in \mathcal{S}^\infty_{[T-h, T]} \times \mathcal{J}^\infty_{[T-h, T]} \times \mathcal{A}_D^\infty$ on the time interval $[T-h, T]$. Next we take $T-h$ as the terminal time and $Y_{T-h}^n$ as the terminal condition. We can find the unique deterministic flat solution of the BSDE (\ref{eq00}) with mean reflection $\left(Y^{n-1}, U^{n-1}, K^{n-1}\right)$  on the time interval $[T-2 h, T-h]$. 
Repeating this procedure, we get a sequence $\left(Y^i, U^i, Z^i, K^i\right)_{i \leq n}$. 

We stitch the sequence as:
$$
Y_t=\sum_{i=1}^n Y_t^i I_{[(i-1) h, i h)}(t)+Y_T^n I_{\{T\}}(t), \ U_t=\sum_{i=1}^n U_t^i I_{[(i-1) h, i h)}(t)+U_T^n I_{\{T\}}(t),
$$
as well as
$$
K_t=K_t^i+\sum_{j=1}^{i-1} K_{j h}^j, \quad \text { for } t \in[(i-1) h, i h],  i \leq n.
$$
 It is  obvious that $(Y, U, K) \in \mathcal{S}^\infty \times \mathcal{J}^\infty \times \mathcal{A}_D^\infty$ is a deterministic flat solution to the BSDE (\ref{eq00}) with mean reflection.
The uniqueness of the global solution follows from the uniqueness of local solution on each small time interval. The proof is complete.
\end{proof}

\begin{remark}
If the BSDE is driven by a marked point process (MPP), and the form is as follows:
\begin{equation}
\label{eq MPP}
\left\{\begin{array}{l}
Y_t =\xi+\int_t^T f\left(s, Y_s, U_s\right) dA_s -\int_t^T \int_E U_s(e) q(d s, d e) + \left(K_T-K_t\right) , \quad \forall t \in[0, T],\quad \mathbb{P} \text {-a.s. };\\
\mathbb{E}\left[l\left(t, Y_t\right)\right] \geq 0, \quad \forall t \in[0, T].
\end{array}\right.
\end{equation}
Here $q$ is a compensated random measure corresponding to some marked point process. The process $A$ is the dual predictive projection of the event counting process related to the marked point process. We emphasize that under our assumptions, as in \cite{foresta2021optimal}, the process $A$ is continuous and increasing which is not necessarily absolutely continuous with respect to the Lebesgue measure. 

Assume process $A$ is continuous, $ \mathbb{E}\left[ e^{\beta A_T} \right] < \infty$, $\left(H_{\xi}\right)$, $\left(H_f\right)$,  $\left(H_l\right)$ and $\left(H^{'}_l\right)$ hold. Then the mean reflected BSDE driven by a marked point process (\ref{eq MPP})  has a unique deterministic flat solution $(Y, U, K)$ in the space $\mathcal{S}^{\infty} \times \mathcal{J}^\infty \times \mathcal{A}_D^\infty$. 

\end{remark}

\subsection{Regularity results on $K$ and $y$}

\begin{prop}
\label{prop regularity}
Let $p \geq 2$ and assume that the same conditions as Theorem \ref{thm well-posedness} and $\left(H^{''}_{l}\right)$ hold . Let $0 \leq t_1 \leq t_2 \leq T$ be such that $t_2 - t_1 \leq 1$. There exists a constant $C$ depending on $p, T, \kappa$, and $h$ such that the following hold:\\
(i) $\forall t_1 \leq s<t \leq t_2, \quad\left|K_t-K_s\right|^p \leq C|t-s|^{p / 2}$.\\
(ii) $\forall t_1 \leq s \leq t \leq t_2, \quad \mathbb{E}\left[\left|y_t-y_s\right|^p\right] \leq C|t-s|$, where $(y, u) \in \mathcal{S}^{\infty} \times J^\infty $ is the solution to the BSDE (\ref{eq samll}).\\
(iii) $\forall t_1 \leq r<s<t \leq t_2, \quad \mathbb{E}\left[\left|y_s-y_r\right|^p\left|y_t-y_s\right|^p\right] \leq C|t-r|^2$, where $(y, u) \in \mathcal{S}^{\infty} \times J^\infty $ is the solution to the BSDE (\ref{eq samll}).
\end{prop}

\begin{proof}
(i). Let us recall that, for any process $y$,
$$
\begin{gathered}
\bar{L}_t(y)=\inf \{x \in \mathbb{R}: \mathbb{E}[l(t, x+X)] \geq 0\}, \\
L_t(y)=\left(\bar{L}_t(y)\right)^{+}=\inf \{x \geq 0: \mathbb{E}[l(t, x+X)] \geq 0\} .
\end{gathered}
$$
Actually, we have
$$
K_s-K_t=\sup _{t \leq r \leq s} L_r\left(Y_s+y_r-y_s\right) .
$$
Indeed, from the representation of the process $K$, we have
$$
\begin{aligned}
K_T-K_t & =\sup _{t \leq r \leq T} L_r\left(y_r\right)=\max \left\{\sup _{s \leq r \leq T} L_r\left(y_r\right), \sup _{t \leq r \leq s} L_r\left(y_r\right)\right\} \\
& =\max \left\{K_T-K_s, \sup _{t \leq r \leq s} L_r\left(y_r\right)\right\} \\
& =\max \left\{K_T-K_s, \sup _{t \leq r \leq s} L_r\left(Y_s-(K_T-K_s)+y_r-y_s\right)\right\} \\
& =\max \left\{K_T-K_s, \sup _{t \leq r \leq s}\left[\bar{L}_r\left(Y_s-(K_T-K_s)+y_r-y_s\right)^{+}\right]\right\} .
\end{aligned}
$$
By the definition of $\bar{L}$, we observe that for all $y \in \mathbb{R}, \bar{L}_t(Y+y)=\bar{L}_t(Y)-y$, so we have
$$
\begin{aligned}
K_T-K_t & =\max \left\{K_T-K_s, \sup _{t \leq r \leq s}\left[\left(K_T-K_s+\bar{L}_r\left(Y_s+y_r-y_s\right)\right)^{+}\right]\right\} \\
& =K_T-K_s+\max \left\{0, \sup _{t \leq r \leq s}\left[\left(K_T-K_s+\bar{L}_r\left(Y_s+y_r-y_s\right)\right)^{+}-(K_T-K_s)\right]\right\} .
\end{aligned}
$$
Note that $\sup _r\left(f(r)^{+}\right)=\left(\sup _r f(r)\right)^{+}=\max \left(0, \sup _r f(r)\right)$ for any function $f$, and obviously
$$
\begin{aligned}
K_T-K_t & =K_T-K_s+\sup _{t \leq r \leq s}\left[\left\{\left(K_T-K_s+\bar{L}_r\left(Y_s+y_r-y_s\right)\right)^{+}-(K_T-K_s)\right\}^{+}\right] \\
& =K_T-K_s+\sup _{t \leq r \leq s}\left[\left(\bar{L}_r\left(Y_s+y_r-y_s\right)\right)^{+}\right] \\
& =K_T-K_s+\sup _{t \leq r \leq s} L_r\left(Y_s+y_r-y_s\right),
\end{aligned}
$$
so
$$
K_s-K_t=\sup _{t \leq r \leq s} L_r\left(Y_s+y_r-y_s\right) .
$$

Hence, from the previous representation of $K_s-K_t$, we deduce the $\frac{1}{2}$-Hölder property of the function $t \longmapsto K_t$. Indeed, since by definition $L_s\left(Y_s\right)=0$, if $t<s$, by using Remark \ref{rmk L}, we have
$$
\begin{aligned}
\left|K_t-K_s\right| & =\sup _{t \leq r \leq s} L_r\left(Y_s+y_r-y_s\right) \\
& =\sup _{t \leq r \leq s}\left[L_r\left(Y_s+y_r-y_s\right)-L_s\left(Y_s\right)\right] \\
& =\kappa \sup _{t \leq r \leq s} \mathbb{E}\left[\left|y_r-y_s\right|\right] + \bar{r} |t-s|,
\end{aligned}
$$
and so
$$
\begin{aligned}
\left|K_t-K_s\right| \leq & C\left\{\mathbb{E}\left[\sup _{t \leq r \leq s}\left|\int_t^r f\left(u, Y_u, u_u\right) d u\right|\right]+\left(\mathbb{E}\left[\sup _{t \leq r \leq s}\left|\int_t^r \int_E u_u(e) \tilde{\mu}(d u, d e)\right|^2\right]\right)^{\frac{1}{2}}\right\} + \bar{r} |t-s|\\
\leq & C\left\{\int_t^s \mathbb{E}\left[\left|f\left(u, 0,0\right)\right|+|Y_u|+\left(\int_E|u_u(e)|^2\nu(de)\right)^{\frac{1}{2}}\right] d u+\left(\mathbb{E}\left[\int_t^s \int_E\left|u_u(e)\right|^2 \nu(d e) d u\right]\right)^{1 / 2}\right\} + \bar{r} |t-s|\\
\leq & C\left\{|t-s| \mathbb{E}\left[\sup _{0\leq u \leq T}|f(u,0,0)|+\sup _{0\leq u \leq T}\left|Y_u\right|\right]+|t-s|^{1 / 2}\left(\mathbb{E}\left[ \|u_s(e)\|_{\mathcal{J}^\infty} \int_E \nu(de) \right]\right)^{1 / 2}\right\} + \bar{r} |t-s|.
\end{aligned}
$$

Therefore, as $\xi$ is bounded and $\|\alpha\|_{\mathcal{S}^{\infty}}<\infty$, it follows from Corollary 3.4 in \cite{gu2024exponential} and Corollary 1 in \cite{morlais2009utility} that
$$
\left|K_t-K_s\right| \leq C|t-s|^{1 / 2} .
$$

When $s<t$, we have the same conclusion.

(ii).
$$
\begin{aligned}
\mathbb{E}\left[\left|y_t-y_s\right|^p\right] &\leq  2^{p-1} \mathbb{E}\left[\left(\int_s^t\left|f\left(u, Y_u,u_u\right)\right| d u\right)^p+\left|\int_s^t \int_E u_u(e) \tilde{\mu}(d u, d e)\right|^p\right] \\
&\leq  C \sup _{0 \leq r \leq t} \mathbb{E}\left[\left(\int_s^r\left|f\left(u,0,0\right)\right|+|Y_u|+\left(\int_E|u_u(e)|^2\nu(de)\right)^{\frac{1}{2}} d u\right)^p+\left|\int_s^r \int_E u_u(e) \tilde{\mu}(d u, d e)\right|^p\right] \\
&\leq C\left\{|t-s|^{p-1} \mathbb{E}\left[\int_s^t\left(|f(u,0,0)|+\left|Y_{u}\right|\right)^p d u\right]+|t-s|^{p / 2}\mathbb{E}\left[\left(\int_s^t\int_E |u_u(e)|^2 \nu(de)d u\right)^{p / 2}\right] \right.
\\& \quad \left.+\mathbb{E}\left[\left(\int_s^t\int_E |u_u(e)|^2 \nu(de)d u\right)^{p / 2}\right] 
+\mathbb{E}\left[\int_s^t \int_E|u_u(e)|^p \nu(de)d u\right]\right\} \\
&\leq C_1 |t-s|^p\mathbb{E}\left[\sup _{0 \leq t \leq T}\left|Y_t\right|^p\right]+C_2 |t-s|^{p-1}\mathbb{E}\left[\int_s^t |f(u,0,0)|^p\right]\\
&\quad +C_3 |t-s|^{p / 2} \mathbb{E}\left[\|u_s(e)\|^p_{\mathcal{J}^\infty}\left(\int_E \nu(de)\right)^{p / 2}\right]+C_4 |t-s| \mathbb{E}\left[\|u_s(e)\|^p_{\mathcal{J}^\infty} \int_E \nu(de) \right] .
\end{aligned}
$$
The third inequality holds because of H\"older's and Burkholder-Davis-Gundy inequality from \cite{kuhn2023maximal}.

Finally, we conclude that there exists a constant $C$, such that
$$
\forall 0 \leq s \leq t \leq T, \quad \mathbb{E}\left[\left|y_t-y_s\right|^p\right] \leq C|t-s| .
$$

(iii). Let $0 \leq r<s<t \leq T$. From the proof of (ii) and Burkholder-Davis-Gundy inequality, we have
$$
\begin{aligned}
& \mathbb{E}\left[\left|y_s-y_r\right|^p\left|y_t-y_s\right|^p\right] \\
& \leq C \mathbb{E}\left[| y_{ s } - y_{ r } | ^ { p } \left\{\mathbb{E}_s\left[\left|\int_s^t f\left(u,Y_u,u_u\right) d s\right|^p\right]+ \mathbb{E}_s\left[ \left| \int_s^t\int_E u_u(e) \tilde{\mu}(du,de) \right|^p \right]  \right\}\right] \\
& \leq C \mathbb{E}\left[\left|y_s-y_r\right|^p\left\{|t-s|\mathbb{E}_s\left[\|u_s(e)\|^p_\infty \int_E \nu(de)\right]\right\}\right] \\
&\leq C|t-s|\mathbb{E}\left[\left|y_s-y_r\right|^p\right]\\
&\leq C|t-s| |s-r|\\
&\leq C|t-r|^2.
\end{aligned}
$$
\end{proof}

\section{Propagation of chaos}
\label{sec 4}
In this section, we study the interacting particle system:
\begin{equation}
\label{eq particle}
\begin{cases}Y_t^i=\xi^i+\int_t^T f^i(u,Y_u^i,U_u^{i,i}) d u-\int_t^T \int_E \sum_{j=1}^N U_u^{i, j}(e) \tilde{\mu}^j(du,de)+K_T^{(N)}-K_t^{(N)}, & \forall 1 \leq i \leq N, \\ \frac{1}{N} \sum_{i=1}^N l\left(t, Y_t^i\right) \geq 0, & 0 \leq t \leq T,\end{cases}
\end{equation}
where ${\xi^i}_{1 \leq i \leq N}$, ${f^i}_{1 \leq i \leq N}$ and ${\tilde{\mu}^i}_{1 \leq i \leq N}$ are independent copies of $\xi, f$ and $\tilde{\mu}$. The augmented natural filtration of the family of $\{\tilde{\mu}^i\}_{1\leq i \leq N}$ is denoted by $\mathscr{F}_t^{(N)}$. Denote by $\mathbb{F}^N:=\left\{\mathcal{F}_t^N\right\}_{t \in[0, T]}$ the completion of the filtration generated by $\left\{\tilde{\mu}^i\right\}_{1 \leq i \leq N}$. Let $\mathcal{T}_t^N$ be the set of $\mathbb{F}^N$ stopping times with values in $[t, T]$. 

This equation is a multidimensional reflected BSDE. A solution to this equation is a process $\left(Y^{(N)}, U^{(N)}, K^{(N)}\right)$ with values in $\left(\mathbb{R}^N, \mathbb{R}^{N \times N}, \mathbb{R}\right), K^{(N)}$ being a nondecreasing random process with $K_0^{(N)}=0$. Such a solution is denoted by $\left(\left\{Y^i, U^i\right\}_{1 \leq i \leq N}, K^{(N)}\right)$ where, for $1 \leq i \leq N, Y^i$ is real valued and $U^i$ takes its values in $\mathbb{R}^{1 \times N}$.

Consistently, the flatness is defined by the following Skorokhod condition:
$$
    \int_0^T \frac{1}{N} \sum_{i=1}^N l\left(t, Y_t^i\right) d K_t=0.
$$

\begin{remark}
The integration in the Skorokhod condition is well-defined as process $K$ is continuous under our condition. Indeed, the solution $K$ is required to be predictable, in contrast the jump from the MPP driver is non-accessible, thus the jump from $K$ is disjoint from the jump from $N$. As a result, if there exists one $i$ such that $N^i$ jump, then process $K$ has no jump. On the other hand, if for all $i$, $N^i$ has no jump and as the reflection boundary $l(t,y)$ is continuous, process $K$ has no jump, either, which implies that the process $K$ has no jump which means that  $K$ is continuous. 
\end{remark}

\subsection{The particle system with constant driver}
In this subsection, we consider the case where the driver does not depend on $Y$. The equation (\ref{eq particle}) can be rewritten as
\begin{equation}
\label{eq particle fix}
\begin{cases}Y_t^i=\xi^i+\int_t^T f^i_u d u-\int_t^T \int_E \sum_{j=1}^N U_u^{i, j}(e) \tilde{\mu}^j(du,de)+K_T^{(N)}-K_t^{(N)}, & \forall 1 \leq i \leq N, \\ \frac{1}{N} \sum_{i=1}^N l\left(t, Y_t^i\right) \geq 0, & 0 \leq t \leq T.\end{cases}   
\end{equation}

In this case, assumption $\left(H_{\xi}\right)$ and $\left(H_f\right)$ can be reduced to

$\left(H_{\xi}'\right)$ The terminal condition $\xi^i$, $1\leq i \leq N$, are $\mathcal{F}_T$-measurable random variable bounded by some constant $M>0$ such that
$$
\mathbb{E}[l(T, \xi^i)] \geq 0 .
$$

$\left(H_f'\right)$ The map $(\omega, t) \mapsto f(\omega, t)$ is $\mathbb{F}$-progressively measurable and for each $t \in[0, T]$, $f_t, 0$ is bounded by some constant $L$, $\mathbb{P}$-a.s.


Before stating the result, let us introduce some further notations. Let us consider the progressively measurable process $\psi^{(N)}$ defined by
$$
\psi_t^{(N)}=\inf \left\{x \geq 0: \frac{1}{N} \sum_{i=1}^N l\left(t, x+y_t^i\right) \geq 0\right\}, \quad 0 \leq t \leq T,
$$
where we set, for all $1 \leq i \leq N$,
$$
y_t^i=\mathbb{E}\left[\xi^i+\int_t^T f^i_u d u \mid \mathscr{F}_t^{(N)}\right], \quad 0 \leq t \leq T .
$$

\begin{thm}
\label{thm well-poedness particle}
Assume $\left(H_{\xi}'\right)$, $\left(H_f'\right)$,  $\left(H_l\right)$ and $\left(H^{'}_l\right)$ hold. Then the multidimensional reflected BSDE (\ref{eq particle fix}) admits a unique flat solution in the product space $\mathcal{S}^\infty \times \mathcal{J}^\infty \times \mathcal{A}^D$.
Moreover, we have, for $1 \leq i \leq N$,
$$
Y_t^i=y_t^i+S_t, \quad 0 \leq t \leq T,
$$
where $S$ is the Snell envelope of the process $\psi^{(N)}$.
\end{thm}
\begin{proof}
Let us start by constructing a solution. We observe that the process $\psi^{(N)}$ can be written as
$$
\psi_t^{(N)}=L_t\left(y_t^1, \ldots, y_t^N\right), \quad 0 \leq t \leq T,
$$
where, for any $P=\left(P^1, \ldots, P^N\right)$, $Q=\left(Q^1, \ldots, Q^N\right)$ in $\mathbb{R}^N$,
$$
L_t(P)=L\left(P^1, \ldots, P^N\right)=\inf \left\{x \geq 0: \frac{1}{N} \sum_{i=1}^N l\left(t, x+P^i\right) \geq 0\right\}.
$$

As pointed out in \cite{briand2021particles}, $L$ is Lipschitz continuous. Precisely,
\begin{equation}
\label{eq m/M}
|L_t(P)-L_t(Q)| \leq \kappa \frac{1}{N} \sum_{j=1}^N\left|P^j-Q^j\right|.
\end{equation}

Besides, the process $\psi^{(N)}$ belongs to $\mathcal{S}^\infty$. 
Indeed, let us set $x_0:=\inf \{x \geq 0: l(t, x) \geq 0\}$ which is finite in view of the assumptions on $l$. We have
\begin{equation}
\label{eq psi}
\left|\psi_t^{(N)}\right|=\left|L_t\left(y_t^1, \ldots, y_t^N\right)-L_t(0)+L_t(0)\right| \leq x_0+ \kappa \frac{1}{N} \sum_j\left|y_t^j\right|. 
\end{equation}

Since $\psi^{(N)}_t$ is in $\mathcal{S}^\infty$, its Snell envelope $S_t = \underset{\tau \in \mathcal{T}_t^N}{\operatorname{ess} \sup }\mathbb{E}\left[\psi_\tau^{(N)} \mid \mathscr{F}_t^{(N)}\right]$ exists and belongs to $\mathcal{S}^\infty$. In fact $S_t$ can be taken as a right continuous $\mathscr{F}^{(N)}$-supermartingale of class (D). Its Doob-Meyer decomposition provides us the existence and uniqueness of $\left(K^{(N)}, M^{(N)}\right)$, with $K^{(N)}$ a non-decreasing process starting from 0 and $M^{(N)}$ an $\mathscr{F}^{(N)}$-martingale such that
\begin{equation}
\label{doob mayer}
S_t=M_t^{(N)}-K_t^{(N)}, \quad 0 \leq t \leq T .
\end{equation}

Thus,
$$
S_t=\mathbb{E}\left[M_T^{(N)} \mid \mathscr{F}_t^{(N)}\right]-K_t^{(N)}=\mathbb{E}\left[S_T+K_T^{(N)} \mid \mathscr{F}_t^{(N)}\right]-K_t^{(N)}.
$$

Define, for $1 \leq i \leq N$,
$$
Y_t^i=y_t^i+S_t, \quad 0 \leq t \leq T.
$$

We have,
$$
\begin{aligned}
Y_t^i & =\mathbb{E}\left[\xi^i+\int_t^T f^i_u d u \mid \mathscr{F}_t^{(N)}\right]+\mathbb{E}\left[S_T+K_T^{(N)} \mid \mathscr{F}_t^{(N)}\right]-K_t^{(N)} \\
& =\mathbb{E}\left[\xi^i+\int_0^T f^i_u d u+K_T^{(N)} \mid \mathscr{F}_t^{(N)}\right]-\int_0^t f^i_u d u-K_t^{(N)} .
\end{aligned}
$$
By the definition of $\psi_t^{(N)}$, we know that $\psi_T^{(N)}=0$, thus $S_T=0$, the last equality holds.

We can apply Theorem 1.1 in \cite{kunita2004representation} , which gave a representation theorem of local martingales, to deduce that for some $U^i$ in $\mathcal{J}^\infty$ 
$$
Y_t^i=\mathbb{E}\left[\xi^i+\int_0^T f^i_u d u+K_T^{(N)}\right]+\int_0^t \int_E \sum_{j=1}^N U_u^{i, j}(e) \tilde{\mu}^j(du,de)-\int_0^t f^i_u d u-K_t^{(N)} .
$$

We verify easily that $\left(\left\{Y^i, U^i\right\}_{1 \leq i \leq N}, K^{(N)}\right)$ is a solution to (\ref{eq particle fix}).
For the constraint, since $l$ is nondecreasing, we have, by definition of $\psi^{(N)}$,
$$
\frac{1}{N} \sum_{i=1}^N l\left(t, Y_t^i\right)=\frac{1}{N} \sum_{i=1}^N l\left(t, y_t^i+S_t\right) \geq \frac{1}{N} \sum_{i=1}^N l\left(t, y_t^i+\psi_t^{(N)}\right) \geq 0 .
$$

Next, we are going to prove that the Skorokhod condition. Since $S$ is the Snell envelope of $\psi^{(N)}$ and $K^{(N)}$ is the associated nondecreasing process, $S_{t}=\psi_{t}^{(N)}, d K_t^{(N)}$-almost everywhere. 
Indeed, in view of equation (\ref{doob mayer}), we have 
$$
K_T^{(N)} - K_t^{(N)} = M_T^{(N)}-S_T-M_t^{(N)}+S_t=M_T^{(N)}-M_t^{(N)}+S_t.
$$
If $d K_t^{(N)} > 0$, then $K_u^{(N)} > K_t^{(N)}$, where $t<u\leq T$. Thus, we can deduce that
$$
\mathbb{E}\left[K_T^{(N)} - K_t^{(N)}\mid \mathscr{F}_t^{(N)} \right] > \mathbb{E}\left[K_T^{(N)} - K_u^{(N)} \mid \mathscr{F}_t^{(N)}\right], \quad \forall u \in (t,T], 
$$
which can implies
$$
\mathbb{E}\left[S_t \mid \mathscr{F}_t^{(N)} \right] > \mathbb{E}\left[S_u \mid \mathscr{F}_t^{(N)}\right], \quad \forall u \in (t,T],
$$
i.e.
$$
\begin{aligned}
\underset{\tau \in \mathcal{T}_t^N}{\operatorname{ess} \sup }\mathbb{E}\left[\psi_\tau^{(N)} \mid \mathscr{F}_t^{(N)}\right]&=\mathbb{E}\left[\underset{\tau \in \mathcal{T}_t^N}{\operatorname{ess} \sup }\mathbb{E}\left[\psi_\tau^{(N)} \mid \mathscr{F}_t^{(N)}\right] \mid \mathscr{F}_t^{(N)} \right]\\
&> \mathbb{E}\left[\underset{\tau \in \mathcal{T}_u^N}{\operatorname{ess} \sup }\mathbb{E}\left[\psi_\tau^{(N)} \mid \mathscr{F}_u^{(N)}\right] \mid \mathscr{F}_t^{(N)}\right]\\
&\geq \underset{\tau \in \mathcal{T}_u^N}{\operatorname{ess} \sup } \mathbb{E}\left[\mathbb{E}\left[\psi_\tau^{(N)} \mid \mathscr{F}_u^{(N)}\right] \mid \mathscr{F}_t^{(N)}\right]\\
&= \underset{\tau \in \mathcal{T}_u^N}{\operatorname{ess} \sup }\mathbb{E}\left[\psi_\tau^{(N)} \mid \mathscr{F}_u^{(N)}\right], \quad \forall u \in (t,T].
\end{aligned}
$$
Thus, we can obtain that $S_t = \underset{\tau \in \mathcal{T}_t^N}{\operatorname{ess} \sup }\mathbb{E}\left[\psi_\tau^{(N)} \mid \mathscr{F}_t^{(N)}\right] = \psi_t^{(N)}, d K_t^{(N)}$-almost everywhere.

Since $K^{(N)}$ is nondecreasing and $\psi^{(N)}$ nonnegative, define set $A = \left\{S_t=0\right\}$, we obtain $\psi_{s}^{(N)} = 0$ for $s \in [t,T]$ on set $A$ which deduce that $ d K_s^{(N)} = 0$ for $s \in [t,T]$ on set $A$. Thus, we have
$$
\begin{aligned}
\int_0^T \frac{1}{N} \sum_{i=1}^N l\left(t, Y_{t}^i\right) d K_t^{(N)} & =\int_0^T \frac{1}{N} \sum_{i=1}^N l\left(t, y_{t}^i+S_{t}\right) \mathbf{1}_{S_{t}>0} d K_t^{(N)} \\
& =\int_0^T \frac{1}{N} \sum_{i=1}^N l\left(t, y_{t}^i+\psi_{t}^{(N)}\right) \mathbf{1}_{\psi_{t}^{(N)}>0} d K_t^{(N)}=0
\end{aligned}
$$
by definition of $\psi^{(N)}$.

Let us turn to the uniqueness. Consider another flat solution $\left(\left\{\tilde{Y}^i, \tilde{U}^i\right\}_{1 \leq i \leq N}, \tilde{K}^{(N)}\right)$. We have
$$
\tilde{Y}_t^i=y_t^i+\mathbb{E}\left[S_T+\tilde{K}_T^{(N)} \mid \mathscr{F}_t^{(N)}\right]-\tilde{K}_t^{(N)}
$$
and, since the constraint is satisfied, by definition of $\psi^{(N)}$, the supermartingale
$$
\mathbb{E}\left[S_T+\tilde{K}_T^{(N)} \mid \mathscr{F}_t^{(N)}\right]-\tilde{K}_t^{(N)}
$$
is bounded from below by the process $\psi_t^{(N)}$. Since $S$ is the Snell envelope of $\psi^{(N)}$, we have
$$
\mathbb{E}\left[S_T+\tilde{K}_T^{(N)} \mid \mathscr{F}_t^{(N)}\right]-\tilde{K}_t^{(N)} \geq S_t, \quad \tilde{Y}_t^i \geq Y_t^i.
$$

Let us suppose that there exists $1\leq i \leq N$ and $0\leq t\leq T$ such that $\mathbb{P}\left(\tilde{Y}_t^i>Y_t^i\right)>0$. Let us consider the stopping time
$$
\tau=\inf \left\{u \geq t: \tilde{Y}_u^i=Y_u^i\right\}.
$$

Then, on the set $\left\{\tilde{Y}_t^i>Y_t^i\right\}, t < \tau \leq T$ and $\tilde{Y}_u^i>Y_u^i$ for $t \leq u < \tau$. Therefore, on this set, since $l$ is increasing,
$$
\sum_{j=1}^N l\left(t, \tilde{Y}_t^j\right)>\sum_{j=1}^N l\left(t, Y_t^j\right) \geq 0
$$
and $d \widetilde{K}^{(N)}_u \equiv 0$ for $t \leq u <\tau$ due to the Skorokhod condition.
We have, by definition of $\tau$,
$$
\begin{aligned}
Y_t^i-\tilde{Y}_t^i & =Y_\tau^i-\tilde{Y}_\tau^i-\int_t^\tau \int_E \sum_{j=1}^N\left(U_u^{i, j}-\widetilde{U}_u^{i, j}\right) \tilde{\mu}^j(du,de)+K_\tau^{(N)}-K_t^{(N)}-\left(\widetilde{K}_\tau^{(N)}-\widetilde{K}_t^{(N)}\right), \\
& =-\int_t^\tau \int_E \sum_{j=1}^N\left(U_u^{i, j}-\widetilde{U}_u^{i, j}\right) \tilde{\mu}^j(du,de)+K_\tau^{(N)}-K_t^{(N)}-\left(\widetilde{K}_\tau^{(N)}-\widetilde{K}_t^{(N)}\right),
\end{aligned}
$$
and, since on the set $\left\{\tilde{Y}_t^i>Y_t^i\right\}, \tilde{K}_t^{(N)}=\tilde{K}_\tau^{(N)}$,
$$
\left(Y_t^i-\tilde{Y}_t^i\right) \mathbf{1}_{\widetilde{Y}_t^i>Y_t^i}=\left(K_\tau^{(N)}-K_t^{(N)}\right) \mathbf{1}_{\widetilde{Y}_t^i>Y_t^i}-\mathbf{1}_{\widetilde{Y}_t^i>Y_t^i} \int_t^\tau \int_E \sum_{j=1}^N\left(U_u^{i, j}-\widetilde{U}_u^{i, j}\right) d \tilde{\mu}^j(du,de) .
$$

Taking the expectation, we obtain
$$
0>\mathbb{E}\left[\left(Y_t^i-\tilde{Y}_t^i\right) \mathbf{1}_{\tilde{Y}_t^i>Y_t^i}\right]=\mathbb{E}\left[\left(K_\tau^{(N)}-K_t^{(N)}\right) \mathbf{1}_{\widetilde{Y}_t^i>Y_t^i}\right] \geq 0,
$$
which is a contradiction. So $\tilde{Y}^i=Y^i$ for all $1 \leq i \leq N$. By uniqueness of the Doob-Meyer decomposition, it follows that $\widetilde{K}^{(N)}=K^{(N)}$ and $\widetilde{U}=U$.
\end{proof}

\begin{remark}
\label{rmk terminal}
Let $0 \leq t \leq T$. For $1 \leq i \leq N$,
$$
Y_s^i=\mathbb{E}\left[Y_t^i+\int_s^t f^i_u d u \mid \mathscr{F}_s^{(N)}\right]+R_s, \quad 0 \leq s \leq t,
$$
where $\left\{R_s\right\}_{0 \leq s \leq t}$ is the Snell envelope of the process
$$
\phi_s^{(N)}=\inf \left\{x \geq 0: \frac{1}{N} \sum_{i=1}^N l\left(s, x+\mathbb{E}\left[Y_t^i+\int_s^t f^i_u d u \mid \mathscr{F}_s^{(N)}\right]\right) \geq 0\right\}, \quad 0 \leq s \leq t.
$$
\end{remark}

\subsection{Propagation of chaos}

In this subsection, we deal with the case where the driver depends on $Y$. Consider equation (\ref{eq particle})

\begin{prop}
\label{prop well-posedness}
Assume assumptions $\left(H_{\xi}'\right)$, $\left(H_f\right)$,  $\left(H_l\right)$ and $\left(H^{'}_l\right)$ hold. Then the reflected BSDE (\ref{eq particle}) admits a unique flat solution in the space $\mathcal{S}^\infty \times \mathcal{J}^\infty \times \mathcal{A}^D$.
\end{prop}

\begin{proof}
We use a fixed point argument. Define a map $\Gamma$ from $(\mathcal{S}^2\times \mathscr{M}^{2,2})^N$ into itself defined by $(Y,U)=\Gamma(P,Q)$ where $(Y, U, K)$ stands for the unique flat solution to
$$
\left\{\begin{array}{l}
Y_t^i=\xi^i+\int_t^T f^i\left(u, P_u^i,Q_u^{i,i}\right) d u-\int_t^T \int_E \sum_{j=1}^N U_u^{i, j} \tilde{\mu}^j(du,de)+K_T^{(N)}-K_t^{(N)}, \quad \forall 1 \leq i \leq N, \quad 0 \leq t \leq T, \\
\frac{1}{N} \sum_{i=1}^N l\left(t, Y_t^i\right) \geq 0, \quad 0 \leq t \leq T .
\end{array}\right.
$$
Let $\left\{Y^i,U^i\right\}_{1 \leq i \leq N}=\Gamma\left(\left\{P^i,Q^i\right\}_{1 \leq i \leq N}\right),\left\{\tilde{Y}^i,\tilde{U}^i\right\}_{1 \leq i \leq N}=\Gamma\left(\left\{\tilde{P}^i,\tilde{Q}^i\right\}_{1 \leq i \leq N}\right)$ and denote by $\Delta \cdot$ the corresponding differences. We have
$$
\left|\Delta Y_t^i\right| \leq\left|\Delta y_t^i\right|+\left|\Delta S_t\right|\leq \left|\Delta y_t^i\right| +\mathbb{E}\left[\sup _{0\leq s \leq T}\left|\Delta \psi_s^{(N)}\right| \mid \mathscr{F}_t^{(N)}\right] ,
$$
where 
$$
y_t^i=\mathbb{E}_t\left[\xi^i+\int_t^T f^i\left(u, P_u^i,Q_u^{i,i}\right) d u\right],
$$
$$
\psi_t^{(N)} = \inf \left\{x \geq 0: \frac{1}{N} \sum_{i=1}^N l\left(t, x+y_t^i\right) \geq 0\right\}, \quad 0 \leq t \leq T.
$$
Applying (\ref{eq m/M}) and the lipschitz condition, we can deduce that
$$
\begin{aligned}
\left|\Delta Y_t^i\right| &\leq \mathbb{E}\left[\int_t^T \left|f^i\left(u, P_u^i,Q_u^{i,i}\right)-f^i\left(u, \tilde{P}_u^i,\tilde{Q}_u^{i,i}\right) \right| du\mid \mathscr{F}_t^{(N)}\right]+\mathbb{E}\left[\sup _{0\leq s \leq T}\left|\Delta \psi_s^{(N)}\right| \mid \mathscr{F}_t^{(N)}\right]\\
&\leq \lambda (T-t) \left(\left\|\Delta P^i \right\|_{\mathcal{S}^\infty}+\left\|\Delta Q^i \right\|_{\mathcal{S}^\infty}\right)+\kappa \mathbb{E}\left[\sup _{0 \leq t \leq T}\left(\frac{1}{N} \sum_{j=1}^N\left|\Delta y_t^j\right|\right)\right]\\
&\leq \lambda T \left(\left\|\Delta P^i \right\|_{\mathcal{S}^\infty}+\left\|\Delta Q^i \right\|_{\mathcal{S}^\infty}\right)+\lambda T \kappa\left[\frac{1}{N} \sum_{j=1}^N \left(\left\|\Delta P^j \right\|_{\mathcal{S}^\infty}+\left\|\Delta Q^j \right\|_{\mathcal{S}^\infty}\right)\right].
\end{aligned}
$$
Thus we have
$$
\left\|\Delta Y^i\right\|_{\mathcal{S}^\infty} \leq \lambda T \left(\left\|\Delta P^i \right\|_{\mathcal{S}^\infty}+\left\|\Delta Q^i \right\|_{\mathcal{S}^\infty}\right)+\lambda T \kappa\left[\frac{1}{N} \sum_{j=1}^N \left(\left\|\Delta P^j \right\|_{\mathcal{S}^\infty}+\left\|\Delta Q^j \right\|_{\mathcal{S}^\infty}\right)\right].
$$
Summing these inequalities gives
$$\begin{aligned}
\frac{1}{N} \sum_{i=1}^N \left\|\Delta Y^i\right\|_{\mathcal{S}^\infty} \leq C_{\lambda,\kappa} T \left[\frac{1}{N} \sum_{i=1}^N \left(\left\|\Delta P^i \right\|_{\mathcal{S}^\infty}+\left\|\Delta Q^i \right\|_{\mathcal{S}^\infty}\right)\right].
\end{aligned}
$$
As process $K^{(N)}$ is continuous, in view of Corollary 1 in \cite{morlais2009utility}, we have
$$
\left\|\Delta U^i\right\|_{\mathcal{J}^\infty}\leq  2\left\|\Delta Y^i\right\|_{\mathcal{S}^\infty}.
$$
Hence,
$$
\frac{1}{N} \sum_{i=1}^N \left\|\Delta Y^i\right\|_{\mathcal{S}^\infty}+\frac{1}{N} \sum_{i=1}^N \left\|\Delta U^i\right\|_{\mathcal{J}^\infty}\leq C_{\lambda,\kappa} T \left[\frac{1}{N} \sum_{i=1}^N \left(\left\|\Delta P^i \right\|_{\mathcal{S}^\infty}+\left\|\Delta Q^i \right\|_{\mathcal{S}^\infty}\right)\right].
$$

It follows that $\Gamma$ has a unique fixed point  as soon as $T$ is small enough: there exists a unique $\left\{Y^i,U^i\right\}_{1 \leq i \leq N}$ solving (\ref{eq particle}). We deduce finally that $K^{(N)}$ is also unique.
\end{proof}

\begin{prop}
\label{estimate of Y}
Assume identically as in Theorem \ref{prop well-posedness}. There exists a constant $C$ independent of $N$ such that, for all $1 \leq i \leq N$,
$$
\left\|Y^i\right\|_{\mathcal{S}^\infty}+\left\|U^i\right\|_{\mathcal{J}^\infty}+\left\|K^{(N)}\right\|_{\mathcal{S}^\infty} \leq C\left(L+ \left\|\xi\right\|_{\infty}\right) .
$$
\end{prop}

\begin{proof}
Let $\left(T_k\right)_{0 \leq k \leq r}$ be a subdivision of $[0, T]$ with $\max _{1 \leq k \leq r}\left(T_k-T_{k-1}\right)=\pi$. We set $I_k=\left[T_{k-1}, T_k\right]$. By Remark \ref{rmk terminal} , for $1 \leq k \leq r$ and $t \in I_k$,
$$
Y_t^i=\mathbb{E}\left[Y_{T_k}^i+\int_t^{T_k} f^i\left(u, Y_u^i,U_u^{i,i}\right) d u \mid \mathscr{F}_t^{(N)}\right]+R_t,
$$
where $\left\{R_t\right\}_{T_{k-1} \leq t \leq T_k}$ is the Snell envelope of the process
$$
\phi_t^{(N)}=\inf \left\{x \geq 0: \frac{1}{N} \sum_{i=1}^N l\left(x+\mathbb{E}\left[Y_{T_k}^i+\int_t^{T_k} f^i\left(u, Y_u^i,U_u^{i,i}\right) d u \mid \mathscr{F}_t^{(N)}\right]\right) \geq 0\right\} .
$$

In view of (\ref{eq psi}), we can obtain
\begin{equation}
\begin{aligned}
\mathbb{E}\left[Y_{T_k}^i+\int_t^{T_k} f^i\left(u, Y_u^i,U_u^{i,i}\right) d u \mid \mathscr{F}_t^{(N)}\right] &\leq \left\|Y_{T_k}^i \right\|_\infty +L+ \pi \lambda \left\|Y^i \right\|_{\mathcal{S}^\infty_{[T_{k-1},T_k]}}+\lambda \mathbb{E}_t\left[\int_t^{T_k} \left(\int_E |U_u^i|^2 \nu(de) \right)^{\frac{1}{2}}du \right]\\
&\leq \left\|Y_{T_k}^i \right\|_\infty +L+ \pi \lambda \left\|Y^i \right\|_{\mathcal{S}^\infty_{[T_{k-1},T_k]}}+\pi^{\frac{1}{2}}\lambda \mathbb{E}_t\left[\int_t^{T_k} \int_E |U_u^i|^2 \nu(de)du \right] ^{\frac{1}{2}}\\
&\leq \left\|Y_{T_k}^i \right\|_\infty +L+ \pi \lambda \left\|Y^i \right\|_{\mathcal{S}^\infty_{[T_{k-1},T_k]}}+\pi \lambda \left\|U_u^i \right\|_{\mathcal{J}^\infty_{[T_{k-1},T_k]}}\\
&\leq A(i,k)+3\pi \lambda \left\|Y^i \right\|_{\mathcal{S}^\infty_{[T_{k-1},T_k]}},
\end{aligned}
\end{equation}
where, for $1 \leq i \leq N$ and $1 \leq k \leq r$,
$$
A(j, k)=\left\|Y_{T_k}^i \right\|_\infty +L,
$$
and the last inequality is deduced by Corollary 1 from \cite{morlais2009utility}.
By the definition of $R_t$, we have
$$
R_t \leq C_\lambda \frac{1}{N} \sum_{j=1}^N\left(A(j, k)+\pi  \left\|Y^i \right\|_{\mathcal{S}^\infty_{[T_{k-1},T_k]}}\right).
$$
Thus,
\begin{equation}
\label{eq supY}
\left\|Y^i \right\|_{\mathcal{S}^\infty_{[T_{k-1},T_k]}} \leq C_\lambda\left(A(j, k)+\pi  \left\|Y^i \right\|_{\mathcal{S}^\infty_{[T_{k-1},T_k]}} \right)+C_\lambda \frac{1}{N} \sum_{j=1}^N\left(A(j, k)+\pi  \left\|Y^i \right\|_{\mathcal{S}^\infty_{[T_{k-1},T_k]}}\right).
\end{equation}
Summing these inequalities gives
$$
\frac{1}{N} \sum_{i=1}^N\left\|Y^i \right\|_{\mathcal{S}^\infty_{[T_{k-1},T_k]}} \leq C_\lambda \frac{1}{N} \sum_{j=1}^N\left(A(j, k)+\pi  \left\|Y^i \right\|_{\mathcal{S}^\infty_{[T_{k-1},T_k]}}\right) .
$$
Let us choose $\pi$ small enough to get, for $1 \leq k \leq r$,
$$
\frac{1}{N} \sum_{i=1}^N\left\|Y^i \right\|_{\mathcal{S}^\infty_{[T_{k-1},T_k]}} \leq C_\lambda \frac{1}{N} \sum_{j=1}^N A(j, k) .
$$
Let us observe that
$$
\begin{aligned}
& A(j, r)=\left\|\xi\right\|_\infty+L, \\
& A(j, k) \leq \left\|Y^i\right\|_{\mathcal{S}^\infty_{[T_{k-1},T_k]}} +L, \quad 1 \leq k \leq r-1 .
\end{aligned}
$$
Thus, for any constant $\alpha>0$,
$$
\begin{aligned}
\sum_{k=1}^r \alpha^k \frac{1}{N} \sum_{i=1}^N\left\|Y^i \right\|_{\mathcal{S}^\infty_{[T_{k-1},T_k]}}
& \leq \alpha^r \frac{C_\lambda}{N} \sum_{j=1}^N A(j, r)+ \sum_{k=1}^{r-1}\alpha^k \frac{C_\lambda}{N} \sum_{j=1}^N A(j, k)\\ 
&\leq C_\lambda \alpha^r\left(\left\|\xi\right\|_\infty+L\right)+\frac{C_\lambda}{\alpha} \sum_{k=1}^r \alpha^k \frac{1}{N} \sum_{i=1}^N\left\|Y^i \right\|_{\mathcal{S}^\infty_{[T_{k-1},T_k]}},
\end{aligned}
$$
and choosing $\alpha>C_\lambda$, we get
$$
\frac{1}{N} \sum_{i=1}^N\left\|Y^i \right\|_{\mathcal{S}^\infty}\leq C\left(\left\|\xi\right\|_\infty+L \right).
$$
With the help of this inequality, we can go back to (\ref{eq supY}) and do the same computation, to get, for $0 \leq i \leq N$,
$$
\left\|Y^i \right\|_{\mathcal{S}^\infty}\leq C\left(\left\|\xi\right\|_\infty+L \right).
$$
In view of  Corollary 1 from \cite{morlais2009utility}, we have
$$
\left\|U^i \right\|_{\mathcal{J}^\infty}\leq 2\left\|Y^i \right\|_{\mathcal{S}^\infty}\leq C\left(\left\|\xi\right\|_\infty+L \right).
$$
Since
$$
K_T^{(N)}=Y_0^i-\xi^i-\int_0^T f^i(u,Y_u^i,U_u^{i,i}) d u+\int_0^T \sum_{j=1}^N \int_E U_u^{i, j}(e) \tilde{\mu}^j(du,de)
$$
we have, with the previous estimate,
$$
\left\|K^{(N)} \right\|_{\mathcal{S}^\infty}\leq C\left(\left\|\xi\right\|_\infty+L \right).
$$
\end{proof}

Let us consider $\left(\bar{Y}^i, \bar{U}^i, K\right)$  independent copies of $(Y, U, K)$ i.e $\left(\bar{Y}^i, \bar{U}^i, K\right)$ is the flat deterministic solution to
$$
\bar{Y}_t^i=\xi^i+\int_t^T f^i\left(u, \bar{Y}_u^i,\bar{U}_u^i\right) d u-\int_t^T \int_E \bar{U}_u^i(e) \tilde{\mu}^i(du,de)+\left(K_T-K_t\right), \quad 0 \leq t \leq T,
$$
with $\mathbb{E}\left[l\left(t, \bar{Y}_t^i\right)\right] \geq 0$ and the Skorokhod condition $\int_0^T \mathbb{E}\left[l\left(t, \bar{Y}_t^i\right)\right] d K_t=0$.

\begin{thm}
Assume assumptions $\left(H_{\xi}\right)$, $\left(H_f\right)$,  $\left(H_l\right)$, $\left(H^{'}_l\right)$ and $\left(H^{''}_l\right)$  hold. Define $ \Delta Y^i:=Y^i-\bar{Y}^i, \Delta K:=K^{(N)}-K$ and $\Delta U^i:=U_t^i-\bar{U}^i e_i$ for $1 \leq i \leq N$, where $\left(e_1, \ldots, e_N\right)$ stand for the canonical basis in $\mathbb{R}^N$. Then
$$
\left\|\Delta Y^i\right\|_{\mathscr{S}^2}^2=\mathcal{O}\left(N^{-1 / 2}\right),\left\|\Delta U^i\right\|_{\mathscr{M}^{2,2}}^2=\mathcal{O}\left(N^{-1 / 4}\right) \text { and }\|\Delta K\|_{\mathscr{A}^2}^2=\mathcal{O}\left(N^{-1 / 4}\right).
$$
\end{thm}

\begin{proof}
Recall Theorem \ref{thm well-poedness particle} and Remark \ref{rmk rep}, for all $1 \leq i \leq N$, we have
$$
\begin{aligned}
& Y_t^i=y_t^i+S_t \text { with } S_t=\underset{\tau \in \mathcal{T}_t^N}{\operatorname{ess} \sup }\mathbb{E}\left[\psi_\tau^{(N)} \mid \mathscr{F}_t^{(N)}\right], \\
& \bar{Y}_t^i=\bar{y}_t^i+R_t \text { with } R_t=\sup _{t\leq s \leq T} \psi_s=\underset{\tau \in \mathcal{T}_t^N}{\operatorname{ess} \sup } \mathbb{E}\left[\psi_\tau \mid \mathscr{F}_t^{(N)}\right],
\end{aligned}
$$
where 
$$
\begin{aligned}
& y_t^i=\mathbb{E}\left[\xi^i+\int_t^T f^i\left(u, Y_u^i,U_u^{i,i}\right) d u \mid \mathscr{F}_t^{(N)}\right],\\
& \psi_t=L_t(y_t)=\inf \left\{x \geq 0: \mathbb{E} \left[l\left(t, x+y_t\right)\right] \geq 0\right\} .
\end{aligned}
$$
Since $\{\bar{Y}^i\}_{1\leq i \leq N}$, $\xi^i$ are independent copies of $Y$, $\xi$,
$$
\begin{aligned}
\bar{y}_t^i=\mathbb{E}\left[\xi^i+\int_t^T f^i\left(u, \bar{Y}_u^i,\bar{U}_u^{i}\right) d u \mid \mathscr{F}_t^i\right] & =\mathbb{E}\left[\xi^i+\int_t^T f^i\left(u, \bar{Y}_u^i,\bar{U}_u^{i}\right) d u \mid \mathscr{F}_t^{(N)}\right], \\
& :=\mathbb{E}\left[\xi^i+\int_t^T \bar{f}_u^i d u \mid \mathscr{F}_t^{(N)}\right].
\end{aligned}
$$
We also define
$$
\bar{\psi}_t^{(N)}=\inf \left\{x \geq 0: \frac{1}{N} \sum_{j=1}^N l\left(t, x+\bar{y}_t^j\right) \geq 0\right\} .
$$

The proof is divided in three steps. In the first step, we show that the convergence rate for the $Y$-component of the solution is given by the convergence rate of $\bar{\psi}^{(N)}$ to $\psi$. In the second step, we study the convergence rate of $\bar{\psi}^{(N)}$ to $\psi$. In the third step, we deduce from the previous steps the convergence rate for the $U$-component and the $K$-component of the solution.

\textbf{Step 1.} Applying Itô's formula to $\Delta \left|\Delta Y_t^i\right|^2$, we obtain, for $t \geq r$,
$$
\begin{aligned}
\left|\Delta Y_t^i\right|^2 & \leq \int_t^T 2 \left|\Delta Y_u^i\right| \left|f^i(u,Y_u^i,U_u^{i,i})-f(u,\bar{Y}_u^i,\bar{U}_u^i)\right|du-\int_t^T\int_E 2 \left|\Delta Y_u^i\right| \left|\Delta U_u^i(e)\right|\tilde{\mu}^j(du,de)\\
&\quad +\int_t^T 2 \left|\Delta Y_u^i\right|d\Delta K_u - \int_t^T\int_E  \left|\Delta U_u^i(e)\right|^2\mu(du,de)\\
&= \mathbb{E}\left[\int_t^T 2 \left|\Delta Y_u^i\right| \left|f^i(u,Y_u^i,U_u^{i,i})-f(u,\bar{Y}_u^i,\bar{U}_u^i)\right|du +\int_t^T 2 \left|\Delta Y_u^i\right|d\Delta K_u - \int_t^T\int_E  \left|\Delta U_u^i(e)\right|^2\nu(de)du \mid \mathscr{F}_t^{(N)}\right] \\
&\leq  \mathbb{E}\left[\int_t^T C_\lambda \left|\Delta Y_u^i\right|^2du+\int_t^T C_\lambda \left|\Delta Y_u^i\right| \left(\int_E (\Delta U_u^i(e))^2\nu(de) \right)^{\frac{1}{2}}du de)du \mid \mathscr{F}_t^{(N)}\right]\\
&\quad + \mathbb{E}\left[\int_t^T 2 \left|\Delta Y_u^i\right|d\Delta K_u de)du \mid \mathscr{F}_t^{(N)}\right]-\mathbb{E}\left[\int_t^T\int_E  \left|\Delta U_u^i(e)\right|^2\nu(de)dude)du \mid \mathscr{F}_t^{(N)}\right]\\
&\leq C_\lambda \mathbb{E}\left[\int_t^T\left|\Delta Y_u^i\right|^2 d u \mid \mathscr{F}_t^{(N)}\right]+\mathbb{E}\left[\sup _{t \leq s \leq T}\left|\psi_s^{(N)}-\psi_s\right|^2 \mid \mathscr{F}_t^{(N)}\right].
\end{aligned}
$$ 
As $S_t = M_t^{N}-K_t^{N}$, the last inequality holds.
Hence,
$$
\mathbb{E}\left[\sup _{r \leq t \leq T}\left|\Delta Y_t^i\right|^2\right]  \leq C_\lambda \left(\mathbb{E}\left[\int_r^T\left|\Delta Y_u^i\right|^2 d u \mid \mathscr{F}_t^{(N)}\right]+\mathbb{E}\left[\sup _{r \leq s \leq T}\left|\Delta \psi_s^{(N)}\right|^2 \mid \mathscr{F}_t^{(N)}\right]+\mathbb{E}\left[\sup _{0 \leq s \leq T}\left|\bar{\psi}_s^{(N)}-\psi_s\right|^2 \mid \mathscr{F}_t^{(N)}\right]\right),
$$
where $\Delta \psi_t^{(N)} = \psi_t^{(N)} -\bar{\psi}_t^{(N)}$.
On the other hand, using (\ref{eq m/M}), when $t \geq r$,
$$
\left|\Delta \psi_t^{(N)}\right| \leq \kappa \frac{1}{N} \sum_{j=1}^N\left|\Delta y_t^j\right| \leq \lambda \kappa \mathbb{E}\left[\frac{1}{N} \sum_{j=1}^N \int_r^T\left|\Delta Y_u^j\right| d u \mid \mathscr{F}_t^{(N)}\right] .
$$
Using Doob's inequality and then Hölder's inequality, we find, for all $0 \leq r \leq T$,
\begin{equation}
\label{estimate sup Yi}
\begin{aligned}
& \mathbb{E}\left[\sup _{r \leq t \leq T}\left|\Delta Y_t^i\right|^2\right] \\
& \quad \leq C_{\lambda,\kappa,T} \left(\mathbb{E}\left[\int_r^T\left|\Delta Y_u^i\right|^2 d u\right]+ \mathbb{E}\left[\frac{1}{N} \sum_{j=1}^N \int_r^T\left|\Delta Y_u^j\right|^2 d u\right]+\left\|\bar{\psi}^{(N)}-\psi\right\|_{\mathscr{S}^2}^2 \right).
\end{aligned}
\end{equation}
Summing these inequalities, we obtain
$$
\begin{aligned}
\mathbb{E}\left[\frac{1}{N} \sum_{i=1}^N \sup _{r \leq t \leq T}\left|\Delta Y_t^i\right|^2\right] & \leq C_{\lambda,\kappa,T} \left( \int_r^T \mathbb{E}\left[\frac{1}{N} \sum_{j=1}^N\left|\Delta Y_u^j\right|^2\right] d u+\left\|\bar{\psi}^{(N)}-\psi\right\|_{\mathscr{S}^2}^2 \right)\\
& \leq C_{\lambda,\kappa,T} \left(  \int_r^T \mathbb{E}\left[\frac{1}{N} \sum_{j=1}^N \sup _{u \leq s \leq T}\left|\Delta Y_s^j\right|^2\right] d u+\left\|\bar{\psi}^{(N)}-\psi\right\|_{\mathscr{S}^2}^2\right).
\end{aligned}
$$
and Gronwall's Lemma gives
$$
\mathbb{E}\left[\frac{1}{N} \sum_{i=1}^N \sup _{0 \leq t \leq T}\left|\Delta Y_t^i\right|^2\right] \leq C_{\lambda,\kappa,T}\left\|\bar{\psi}^{(N)}-\psi\right\|_{\mathscr{S}^2}^2 .
$$

Coming back to the estimate (\ref{estimate sup Yi}), we finally deduce that
$$
\left\|\Delta Y^i\right\|_{\mathscr{S}^2}^2 \leq C_{\lambda, \kappa, T}\left\|\bar{\psi}^{(N)}-\psi\right\|_{\mathscr{S}^2}^2 .
$$
\textbf{Step 2.} For all $0 \leq t \leq T$, denote $\nu_t$ the common law of the random variables $\left\{\bar{y}_t^i\right\}_{1 \leq i \leq N}$ and their empirical law $\nu_t^{(N)}:=\frac{1}{N} \sum_{i=1}^N \delta_{\bar{y}_t^i}$.

Let us define $H:(x, \mu) \in \mathbb{R} \times \mathcal{P}_1(\mathbb{R}) \mapsto \int l(t, x+y) \mu(d y)$. For each probability measure $\mu$, $x \longmapsto H(x, \mu)$ is nondecreasing and bi-Lipschitz with the same constants as $l$. Let us also introduce
$$
\psi_t^*=\inf \left\{x \in \mathbb{R}: \mathbb{E}\left[l\left(t, x+\bar{y}_t^i\right)\right] \geq 0\right\}
$$
and $\bar{\psi}_t^{(N)^*}$ defined in the same way. Since $l$ is continuous, one has
$$
H\left(\psi_t^*, \nu_t\right)=H\left(\bar{\psi}_t^{(N)^*}, \nu_t^{(N)}\right)=0 .
$$

Of course, $\left|\psi_t-\bar{\psi}_t^{(N)}\right| \leq\left|\psi_t^*-\bar{\psi}_t^{(N)^*}\right|$ so that
\begin{equation}
\label{phi}
\begin{aligned}
\left\|\bar{\psi}^{(N)}-\psi\right\|_{\mathscr{S}^2}^2 & \leq \frac{1}{\underline\kappa^2} \mathbb{E}\left[\sup _{0 \leq t \leq T}\left|H\left(\bar{\psi}^{(N)^*}, \nu_t^{(N)}\right)-H\left(\psi_t^*, \nu_t^{(N)}\right)\right|^2\right], \\
& \leq \frac{1}{\underline\kappa^2} \mathbb{E}\left[\sup _{0 \leq t \leq T}\left|H\left(\psi_t^*, \nu_t^{(N)}\right)-H\left(\psi_t^*, \nu_t\right)\right|^2\right] .
\end{aligned}
\end{equation}

Since $l$ is Lipschitz, we deduce from (\ref{phi}) that
$$
\left\|\bar{\psi}^{(N)}-\psi\right\|_{\mathscr{S}^2}^2 \leq \kappa^2 \mathbb{E}\left[\sup _{0\leq t \leq T} W_1^2\left(\nu_t^{(N)}, \nu_t\right)\right] .
$$

We recall the bound of the right hand side of the previous inequality from \cite{briand2021particles,fournier2015rate}. Based on the regularity in Proposition \ref{prop regularity} and Proposition \ref{estimate of Y}, we have:
$$
\left\|\bar{\psi}^{(N)}-\psi\right\|_{\mathscr{S}^2}^2 \leq \frac{C}{\sqrt{N}}
$$
with $C$ depending on $p$ and all parameters.

Hence, we deduce that $\left\|\Delta Y^i\right\|_{\mathscr{S}^2}^2=\mathcal{O}\left(N^{-1 / 2}\right)$.

\textbf{Step 3.}  We know that $\left(\Delta Y^i, \Delta Z^i, \Delta K\right)$ verifies
$$
\Delta Y_t^i=\int_t^T\left(f^i\left(u, Y_u^i,U_u^{i,i}\right)-f^i\left(u, \bar{Y}_u^i,\bar{U}_u^i\right)\right) d u-\int_t^T \int_E \sum_{j=1}^N \Delta U_u^{i, j}(e) \tilde{\mu}(du,de)+\Delta K_T-\Delta K_t, \quad \forall t \in[0, T].
$$

Use It$\hat o$'s formula, we can deduce that
$$
\begin{aligned}
\mathbb{E}\left[\int_0^T \int_E \sum_{j=1}^N\left|\Delta U_u^{i, j}\right|^2 d u\right]\leq C_{ \lambda, T}\left\|\Delta Y^i\right\|_{\mathscr{S}^2}^2+C\left\|\Delta Y^i\right\|_{\mathscr{S}^2}\left(\mathbb{E}\left[\left|K_T^{(N)}\right|^2\right]+\left|K_T\right|^2\right)^{1 / 2}.
\end{aligned}
$$

In view of Proposition \ref{estimate of Y} thus we obtain the rate of convergence
$$
\left\|\Delta U^i\right\|_{\mathscr{M}^{2,2}}^2=O\left(\left\|\Delta Y^i\right\|_{\mathscr{S}^2}\right) .
$$

Finally, let us write
$$
\Delta K_t=\Delta Y_0^i-\Delta Y_t^i-\int_0^t\left(f^i\left(u, Y_u^i,U_u^{i,i}\right)-f^i\left(u, \bar{Y}_u^i,\bar{U}_u^i\right)\right) d u+\int_0^t \int_E \sum_{j=1}^N \Delta U_u^{i, j}(e) \tilde{\mu}(du,de) \quad \forall t \in[0, T]
$$
to obtain
$$
\begin{aligned}
\|\Delta K\|_{\mathscr{A}^2}^2 & \leq C\left\|\Delta Y^i\right\|_{\mathscr{S}^2}^2+C\left\|\Delta U^i\right\|_{\mathscr{M}^{2,2}}^2+C_T \mathbb{E}\left[\int_0^T\left|f^i\left(u, Y_u^i,U_u^{i,i}\right)-f^i\left(u, \bar{Y}_u^i,\bar{U}_u^i\right)\right|^2 d u\right] \\
& \leq C_{\lambda, T}\left\|\Delta Y^i\right\|_{\mathscr{S}^2}^2+C\left\|\Delta U^i\right\|_{\mathscr{M}^{2,2}}^2=O\left(\left\|\Delta Y^i\right\|_{\mathscr{S}^2}\right) .
\end{aligned}
$$

This ends the proof of our main result.

\end{proof}

\bibliographystyle{plain}
\bibliography{RBSDEJ}

\end{document}